\documentclass[12pt]{article}
\usepackage{amsmath}
\usepackage{amssymb}
\usepackage{amsthm,mathtools}
\usepackage[mathscr]{eucal}
\usepackage{xcolor}
\usepackage{fontenc}
\usepackage{graphicx}
\usepackage{geometry}
\theoremstyle{definition}
\newtheorem{definition}{Definition}[section]
\newtheorem{theorem}[definition]{Theorem}
\newtheorem*{theorem*}{Conjecture}

\newtheorem{lemma}[definition]{Lemma}
\newtheorem{corollary}[definition]{Corollary}
\theoremstyle{remark}
\newtheorem{remark}[definition]{Remark}
\newtheorem{example}[definition]{Example}
\newcounter{enumctr}

\providecommand{\keywords}[1]{\textbf{\textbf{Key words: }} #1}
\begin{document}
\title {Asymptotic behaviour of solutions to non-commensurate fractional-order planar systems}         
\author{Kai Diethelm\thanks{Faculty of Applied Natural Sciences and Humanities (FANG), 
University of Applied Sciences W\"urzburg-Schweinfurt, Ignaz-Sch\"on-Str.\ 11, 97421 Schweinfurt, Germany, 
\tt kai.diethelm@fhws.de}
\and
Ha Duc Thai\footnote{Institute of Mathematics, Vietnam Academy of Science and Technology, 
18 Hoang Quoc Viet, 10307 Ha Noi, Vietnam, \tt hdthai@math.ac.vn}
\and
Hoang The Tuan\footnote{Institute of Mathematics, Vietnam Academy of Science and Technology, 
18 Hoang Quoc Viet, 10307 Ha Noi, Vietnam, \tt httuan@math.ac.vn}
}
\date{}
\maketitle

\begin{abstract}
This paper is devoted to studying non-commensurate fractional order planar systems. Our contributions are to derive sufficient conditions for the global attractivity of non-trivial solutions to fractional-order inhomogeneous linear planar systems and for the Mittag-Leffler stability of an equilibrium point to fractional order nonlinear planar systems. To achieve these goals, our approach is as follows. Firstly, based on Cauchy's argument principle in complex analysis, we obtain various explicit sufficient conditions for the asymptotic stability of linear systems whose coefficient matrices are constant. Secondly, by using Hankel type contours, we derive some important estimates of special functions arising from a variation of constants formula of solutions to inhomogeneous linear systems. Then, by proposing new weighted norms combined with the Banach fixed point theorem for appropriate Banach spaces, we get the desired conclusions. Finally, numerical examples are provided to illustrate the effect of the main theoretical results.
\end{abstract}



\section{Introduction}
Fractional calculus and fractional order differential equations are research topics that have generated a great amount of interest in recent years. For details on their various applications in in science and engineering, we refer the interested reader to the collections \cite{Baleanu_1, Baleanu_2, Petras, Tarasov_1, Tarasov_2} and the references therein.

To our knowledge, the first contribution in the qualitative study of the fractional order autonomous linear systems was published by Matignon \cite{Matignon}. In that paper, using Laplace transform and the final value theorem, the author has obtained an algebraic criterion to ensure the attractiveness of solutions. The BIBO (bounded input, bounded output) stability for non-commensurate fractional order systems, i.e.\ for systems whose differential equations are not all of the same order, was investigated by Bonnet and Partington \cite{Bonnet}, and their result shows that the systems are stable if and only if their transfer function has no pole in the closed right hand side of the complex plane. 

Starting from \cite{Bonnet}, a new difficult task appears: finding the conditions to ensure that the poles of the characteristic polynomial of the system lie on the open left side of the complex plane. Trigeassou et al.\ \cite{Trige} have proposed a method based on Nyquist’s theorem. In particular, they have derived Routh-like stability conditions for fractional order systems involving at most two fractional derivations. Unfortunately, for higher numbers of differential operators, this approach seems to be unsuitable by its numerical implementation. After that, Sabatier et al.\ \cite{Sabatier} have introduced another realization of the fractional system. This realization is recursively defined and involves nested closed-loops. Based on this realization, they have obtained a recursive algorithm that involves, at each step, Cauchy’s argument principle on a frequency range and removes the numerical limitation mentioned in \cite{Trige} above.

In addition to the algorithmic approach as in \cite{Sabatier}, a number of analytic approaches have been used to investigate the zeros of characteristic polynomials of systems of fractional order systems. In \cite{Ivanova}, the stability and resonance conditions are established for fractional systems of second order in terms of a pseudo-damping factor and a fractional differentiation order. The method in \cite{Ivanova} has been successfully extended in \cite{Chen} for a wide class of second kind non-commensurate elementary systems. By the substitution method, a variation of constants formula and the properties of the Mittag-Leffler function in the stable domain, in \cite{Tuan2017}, the authors have shown the asymptotic stability for fractional order systems with (block) triangular coefficient matrices. By combining a variation of constants formula, properties of Mittag-Leffler functions, a special weighted norm type and Banach's fixed-point theorem, Tuan and Trinh \cite{TH2020} have proved the global attractivity and asymptotic stability for a class of mixed-order linear fractional systems when the coefficient matrices are strictly diagonally dominant and the elements on the main diagonal of these matrices are negative. Using the positivity of the system and developing a novel comparison principle, Shen and Lam \cite{Shen16} have considered the stability and performance analysis of positive mixed fractional order linear systems with bounded delays. Tuan et al.\ \cite{THL21} have established a necessary and sufficient condition for the asymptotic stability of positive mixed fractional-order linear systems with bounded or unbounded time-varying delays.

Although there have been some articles on mixed fractional order systems as listed above, in our view, the qualitative theory of non-commensurate fractional order systems is still a challenging topic whose development is in its infancy. Even in the simplest case when the coefficient matrix is constant, the current results seem to be far away from a complete characterization of the stability of these systems. In particular, the entire theory for non commensurate systems is far less well developed than the corresponding theory for commensurate systems (i.e.\ systems all of whose associated differential equations are of the same order) that have been extensively discussed, e.g., in the papers mentioned above or in \cite{Cong2016,Cong2020} and the references cited therein.

For these reasons, we study in this paper the fractional-order planar system with Caputo fractional derivatives
\begin{align}\label{intro}
	^CD^{\alpha}_{0^+}x(t)&=Ax(t)+f(t,x(t)), \quad t>0,\\ 
 	x(0)&= x^0 \in \mathbb R^2, 
\end{align}
where $\alpha = (\alpha_1,\alpha_2) \in (0,1]^2$ is a multi-index, $A\in\mathbb{R}^{2\times 2}$ is a square matrix and $f :[0, \infty) \times \mathbb R^2 \rightarrow \mathbb R^2$ is vector valued continuous function. It is worth noting that for the case $f=0$, in \cite{Kaslik2021}, by constructing a smooth parameter curve and using Rouch\'e’s theorem, Brandibur and Kaslik have provided criteria for the asymptotic stability and for the instability of solutions, respectively. However, these conditions are not explicit and are quite difficult to verify. Motivated by \cite{Kaslik2021}, our aim is as follows. First, we want to give sufficient simple and clear conditions that can guarantee the Mittag-Leffler stability of the system \eqref{intro} in the homogeneous case. Then, by establishing a variation of constants formula, estimates for general Mittag-Leffler type functions, and proposing new weighted norms, we show the asymptotic behavior of the system when the vector field $f$ is inhomogeneous or represents small nonlinear noise around its equilibrium point.  

The paper is organized as follows. Section \ref{sec:preliminaries} contains a brief summary of existence and uniqueness results for solutions to multi-order fractional differential systems and a variation of constants formula for solutions to fractional order inhomogeneous linear planar systems. Section \ref{ss} deals with some properties of the characteristic function to a general fractional order homogeneous linear planar system whose coefficient matrix is constant. Section \ref{sec:specfun} is devoted to studying important estimates for special functions arising from the variation of constants formula for the solutions. Our main contributions are presented in Section \ref{sec:mainresults} where we show the asymptotic behaviour of solutions to fractional-order linear planar systems and the Mittag-Leffler stability of an equilibrium point to fractional nonlinear planar systems. Numerical examples are provided in Section \ref{sec:examples} to illustrate the main theoretical results.

To conclude the introduction, we present some notations that will be used throughout the rest of the paper. In $\mathbb R^2$, we define the norm $\|\cdot\|$ by $\|x\| := \max\{|x_1|,|x_2|\}$ for every $x\in\mathbb R^2$.
For any $r>0$, the closed ball of radius $r$ centered at the origin $0$ in $\mathbb{R}^2$ is given by $B(0,r):=\{x\in\mathbb R^2:\|x\|\leq r\}$. The space of all continuous functions $\xi:[0,\infty)\rightarrow \mathbb R^2$ is denoted by $C([0,\infty);\mathbb{R}^2)$. For any $\xi\in C([0,\infty);\mathbb{R}^2)$, let $\|\xi\|_\infty:=\sup_{t\geq 0}\|\xi(t)\|$. Then, we use the notation $ C_\infty([0,\infty);\mathbb{R}^2) := \{\xi\in C([0,\infty);\mathbb{R}^2): \|\xi\|_\infty<\infty\}$ to designate the subspace of $C([0,\infty);\mathbb{R}^2)$ that comprises the bounded continuous functions on $[0, \infty)$.

For $\alpha \in (0,1)$ and $J = [0, T]$ or $J = [0, \infty)$, we define the Riemann-Liouville fractional integral of a function $f :J \rightarrow \mathbb R$ as 
$$ I^\alpha_{0^+}f(t) := \frac{1}{\Gamma(\alpha)}\int_{0}^{t}(t-s)^{\alpha -1}f(s)ds,\;t\in J,$$
and the Caputo fractional derivative of the order $\alpha\in (0,1)$ of a function $f : J \rightarrow \mathbb R$ as 
$$ ^C D^\alpha_{0^+}f(t) := \frac{d}{dt}I^{1-\alpha}_{0^+}(f(t)-f(0)), \;t\in J \setminus \{ 0 \},$$
where $\Gamma(\cdot)$ is the Gamma function and $\frac{d}{dt}$ is the usual derivative.
Letting $\alpha = (\alpha_1,\alpha_2) \in (0,1]\times (0,1]$ be a multi-index and $f = (f_1, f_2)$ with $f_i : J \rightarrow \mathbb R$, $i=1,2$, be a vector valued function, we write
$$ ^C D^\alpha_{0^+} f(t) := \left(^C D^{\alpha_1}_{0^+}f_1(t),^C D^{\alpha_2}_{0^+}f_2(t) \right).$$
See, e.g., \cite[Chapter III]{Kai} and \cite{Vainikko_16} for more details on the Caputo fractional derivative.

\section{Preliminaries}
\label{sec:preliminaries}

\subsection{Existence and uniqueness of global solutions and exponential boundedness of solutions}

Consider the two-component incommensurate fractional-order initial value problem with Caputo fractional derivatives
\begin{subequations}
\label{Bt}
\begin{align}
	 ^C D^{\alpha}_{0^+}x(t) &= f(t,x(t)), \quad t>0,\\ 
	 x(0) &= x^0 \in \mathbb R^2, 
\end{align}
\end{subequations}
where $\alpha = (\alpha_1,\alpha_2) \in (0,1]^2$ is a multi-index and $f :[0, \infty) \times \mathbb R^2 \rightarrow \mathbb R^2$ is 
a continuous function.

\begin{theorem}[Existence and uniqueness of global solutions]\label{DLTTDN}
	Suppose that the function $f:[0, \infty) \times \mathbb R^2 \rightarrow \mathbb R^2$ is continuous and that, 
	for some constant $L > 0$, it satisfies the Lipschitz condition
	$$ \|f(t,x) - f(t,\hat x)\| \leq L||x - \hat x||,\; \forall t \in [0.\infty),\; x ,\hat x \in \mathbb R^2$$
	with respect to its second variable. Then, for any initial value $x^0 \in \mathbb R^2$, 
	the two-component incommensurate fractional-order system \eqref{Bt} has 
	a unique global solution $\varphi(\cdot,x^0)$ on the interval $[0,\infty)$. 
\end{theorem}

\begin{proof} 
	See \cite[Theorem 2.2 and Remark 2.3]{Tuan2020}.
\end{proof}

\begin{theorem}[Exponential boundedness of global solutions]\label{DLTBCM}
	Suppose that the function $f$ satisfies the assumptions of Theorem \ref{DLTTDN}.  
	Moreover, let there exist a constant $\gamma > 0$ such that 
	$$ \sup_{t\geq 0} e^{-\gamma t} \int_0^t(t-s)^{\alpha_i -1}\|f(s,0)\|ds < \infty. $$
	Then, for any initial value $x^0 \in \mathbb R^2,$ the two-component incommensurate fractional-order system
	 \eqref{Bt} has a unique global solution $\varphi(\cdot,x^0) \in C\left ( [0,\infty),\mathbb R^2 \right )$ and 
	$$ \|\varphi(t,x^0)\| \leq Me^{\gamma t},\; \forall t \geq 0, $$ 
	where $M$ is some positive constant which depends on $x^0.$
\end{theorem}

\begin{proof} 
	See \cite[Theorem 2.4]{Tuan2020}. 
\end{proof} 


\subsection{The variation of constants formula for the solutions}

Consider the non-homogeneous two-component incommensurate fractional-order linear system
\begin{subequations}
\label{eq:lin-ivp}
\begin{align}\label{Bt1}
 	^CD^{\alpha}_{0^+}x(t)=Ax(t) + f(t) ,t>0 
\end{align}
with initial condition
\begin{align}\label{DKD}
	 x(0) = x^0 \in \mathbb R^2, 
\end{align}
\end{subequations}
where $\alpha = (\alpha_1, \alpha_2) \in (0,1]^2$, 
$A = \left ( a_{ij} \right ) \in \mathbb R^{2 \times 2}$ is a square real matrix and $f=(f_1,f_2)^{\rm T} :[0,\infty) \rightarrow \mathbb R^2$ is a continuous function such that
\begin{equation} 
	\|f(t)\| \leq Me^{\gamma t},\;\forall t \geq 0
\end{equation}
for some $M > 0$ and some $\gamma > 0$. Then, we have
\begin{align*}
	\int_0^t(t-s)^{\alpha_i -1}|f_i(s)|ds  &\leq M \int_0^t(t-s)^{\alpha_i -1}e^{\gamma s}ds\\ 
	 &= -\frac{M e^{\gamma t}}{\gamma^{\alpha_i}}  
	 		\int_0^t\left (\gamma(t-s)  \right )^{\alpha_i -1}e^{-\gamma (t-s)}d(\gamma(t-s))\\ 
	 &= \frac{M e^{\gamma t}}{\gamma^{\alpha_i}}\int_0^{\gamma t} \tau ^{\alpha_i -1}e^{-\tau}d\tau\\ 
	 & \leq \frac{M \Gamma(\alpha_i)}{\gamma ^{\alpha_i}} e^{\gamma t}.
\end{align*} 
Due to Theorems \ref{DLTTDN} and \ref{DLTBCM}, for any initial condition
$x^0 \in \mathbb R^2$, the system \eqref{eq:lin-ivp} has a unique exponentially bounded solution in $C\left ( [0, \infty), \mathbb R^2 \right )$. 
Taking Laplace transform on both sides of the system \eqref{eq:lin-ivp}, we obtain the  algebraic system
\begin{align}
	 \begin{cases}
	(s^{\alpha_1} -a_{11})X_1(s) - a_{12}X_2(s) &= s^{\alpha_1 -1 }x_1^0 +F_1(s)  \\ 
	 -a_{21}X_1(s) +(s^{\alpha_2} -a_{22})X_2(s)& =  s^{\alpha_2 -1 }x_2^0 +F_2(s)
	\end{cases} , 
\end{align}
where $X_i(s)$ and $F_i(s)$, $i=1,2$, are the Laplace transforms of $x_i(t)$ and $f_i(t)$, respectively.
By Cramer's rule, we see that
\begin{align}
	 X_1(s) &= \frac{ x_1^0(s^{\alpha_1 + \alpha_2-1} - a_{22}s^{\alpha_1 -1})+ x^0_2a_{12}s^{\alpha_2-1}
	 + F_1(s)(s^{\alpha_2}-a_{22}) + a_{12}F_2(s) }{Q(s)}\nonumber \\ 
	 &= \frac{s^{\alpha_1 + \alpha_2} - a_{22}s^{\alpha_1}}{sQ(s)}x^0_1 + \frac{s^{\alpha_2}}{sQ(s)}x^0_2 + 
	\frac{s^{\alpha_2}-a_{22}}{Q(s)}F_1(s ) + \frac{a_{12}F_2(s)}{Q(s)},
\end{align}
and
\begin{align}
	 X_2(s) &= \frac{x_2^0(s^{\alpha_1 + \alpha_2-1} - a_{11}s^{\alpha_2 -1})+ x^0_1a_{21}s^{\alpha_1-1} + a_{12}F_1(s) 
	 + F_2(s)(s^{\alpha_1}-a_{11})}{Q(s)}\nonumber \\ 
	 &= \frac{s^{\alpha_1 + \alpha_2} - a_{11}s^{\alpha_2 }}{sQ(s)}x_2^0 + \frac{a_{21}s^{\alpha_1}}{sQ(s)}x_1^0 + 
	\frac{a_{21}F_1(s)}{Q(s)} + \frac{s^{\alpha_1}-a_{11}}{Q(s)}F_2(s),
\end{align}
where $Q(s) := s^{\alpha_1 + \alpha_2} - a_{11}s^{\alpha_2}-a_{22}s^{\alpha_1} +\det A. $
Put 
\begin{subequations}
\label{eq:def-r-s}
\begin{align}
	\mathcal R^\lambda(t) &= \mathcal L^{-1}\left \{\frac{s^{l(\alpha)-\lambda}}{sQ(s)}  \right \}(t) , 
		\quad \lambda \in \left \{ 0,\alpha_1, \alpha_2 \right \}, \\ 
	\mathcal S^\beta(t) &= \mathcal L^{-1}\left \{\frac{s^{l(\alpha)-\beta}}{Q(s)}  \right \}(t) , 
		\quad \beta \in \left \{ \alpha_1, \alpha_2 ,l(\alpha)\right \}
\end{align}
\end{subequations}
with $l(\alpha): = \alpha_1 + \alpha_2.$ Then, with each $i \in \left \{ 1,2 \right \}$, we obtain
\begin{align*}
\mathcal L^{-1}\left \{ \frac{s^{l(\alpha)-\beta}}{Q(s)}F_i(s) \right \}(t) &= \mathcal L^{-1}\left \{ \mathcal L\left \{ \mathcal S^\beta \right \}(s) \mathcal L\left \{ f_i \right \}(s)\right \}(t)\\ 
 &= \mathcal L^{-1} \left\{ \mathcal L\left \{ \mathcal S^\beta \ast f_i \right \}(s)  \right \}(t) \\
&=\mathcal S^\beta \ast f_i(t), \beta \in \left \{ \alpha_1, \alpha_2, l(\alpha) \right\}, 
\end{align*}
where "$\ast$" is the Laplace convolution operator.

From the arguments above, the unique solution to the initial value problem \eqref{eq:lin-ivp} has the following form.

\begin{lemma}\label{BTHS1} 
	On the interval $[0, \infty)$, the non-homogeneous linear two-component incommensurate fractional-order system \eqref{eq:lin-ivp} 
	has the unique solution 
	\[
		\varphi(\cdot,x^0) = \begin{pmatrix}
						\varphi_1(\cdot,x^0)\\
						\varphi_2(\cdot,x^0)
						\end{pmatrix}
	\]
	with
	\begin{align}
		\varphi_1(t,x^0) &=  \left (\mathcal R^0(t)-a_{22}\mathcal R^{\alpha_2}(t)  \right )x_1^0 
						+ a_{12}\mathcal R^{\alpha_1}(t)x^0_2 \nonumber \\
					& \hspace{1cm}+ \left( ( \mathcal S^{\alpha_1} - a_{22}\mathcal S^{l(\alpha)} )*f_1  \right )(t) 
								+ a_{12}\left (\mathcal S^{l(\alpha)}*f_2  \right )(t), \\
		\varphi_2(t,x^0) &=  a_{21}\mathcal R^{\alpha_2}(t)x^0_1 + \left (\mathcal R^0(t)
						- a_{11}\mathcal R^{\alpha_1}(t)  \right )x_2^0 \nonumber \\
					&\hspace{1cm}+a_{21}\left (\mathcal S^{l(\alpha)}*f_1  \right )(t) +
							\left (( \mathcal S^{\alpha_2}- a_{11}\mathcal S^{l(\alpha)})*f_2  \right )(t). 
	\end{align}
\end{lemma}

\section{Some properties of the characteristic function}\label{ss}

In this paper, we only focus on incommensurate systems, i.e.\ on systems
of the form \eqref{intro} with $\alpha_1 \ne \alpha_2$, because the case $\alpha_1 = \alpha_2$
has already been discussed in detail elsewhere \cite{Matignon}. 
Thus, without loss of generality, we assume $0 <\alpha_1 < \alpha_2\leq 1$. 
Our first auxiliary statement in this context deals with functions of the form
\begin{equation}
	\label{eq:def-Q}
	Q(s) = s^{\alpha_1 + \alpha_2} - as^{\alpha_2}- bs^{\alpha_1} +c;
\end{equation}
the characteristic functions of the problems under consideration will be of precisely this structure.

\begin{lemma}\label{bd3.1}
	Let $0 <\alpha_1 < \alpha_2\leq 1$ and $a,b,c \in \mathbb R$. Then, the following
	statements hold for the function $Q$ defined in \eqref{eq:def-Q}.
	\begin{itemize}
	\item [(i)] If $ c < 0$ then $Q$ has at least one positive real zero. 
	\item [(ii)] If $s \in \mathbb C$ is a zero of $Q$ then its complex conjugate is also a zero of $Q$.
	\item [(iii)] Let $0 < \omega < \pi$. Then, $Q$ has only a finite number of zeros in the set 
		$\mathcal C = \{z \in \mathbb C : | \arg{(z)}| \leq \omega\}$. 
	\item [(iv)] If $c > 0$, then $s = i\omega$ with $\omega > 0$ is a zero of $Q$ if and only if 
		\begin{align}\label{eq} &\begin{cases}
			 & a= \rho_2 \omega^{\alpha_1} -c\rho_1 \omega^{-\alpha_2}, \\ 
			 &   b = c\rho_2 \omega^{-\alpha_1} -\rho_1 \omega^{\alpha_2},
			\end{cases}
		\end{align}
		where 
		\begin{equation}
			\label{eq:def-rho}
			\rho_1 = \frac{\sin\frac{\alpha_1 \pi}{2}}{\sin\frac{(\alpha_2 - \alpha_1)\pi}{2}} , \hspace{2cm}
			\rho_2 = \frac{\sin\frac{\alpha_2 \pi}{2}}{\sin\frac{(\alpha_2 - \alpha_1)\pi}{2}}.
		\end{equation}
	\end{itemize}
\end{lemma}

\begin{proof} 
	(i) and (ii) are obvious.
	
	(iii) First, we assume that $c\neq 0$. 
	Then $Q(0) \neq 0$. Due to the continuity of $Q$ at $0$, we can find $\varepsilon$ which is small enough such that 
	$Q$ has no zero in $\left\{z\in \mathbb C: |z| < \varepsilon \right\}.$  Moreover,
	because $|Q(s)| \geq |s|^{\alpha_1 + \alpha_2} - |a| \cdot |s|^{\alpha_2}-|b| \cdot|s|^{\alpha_1} -|c|,$ 
	we have that $\lim_{|s| \to \infty} |Q(s)| = \infty$ uniformly for all $\arg s$. This implies that there is a positive real number $R$ 
	such that $Q$ has no zero in the domain $\left\{z\in\mathbb C: |z| > R \right\}.$ Hence, all zeros of $Q$ in 
	$\{z\in\mathbb C: | \arg{(z)}| \leq \omega\}$ (if they exist) belong to the set 
	$\Omega := \left \{z\in\mathbb C: \varepsilon\leq |z| \leq R, |\arg{(z)}|\leq \omega\right \}.$ 
	Notice that $\Omega$ is a compact set and $Q$ is analytic on this domain. If now $Q$ has infinitely
	many zeros in $\Omega$ then, because of the compactness of $\Omega$, the set of zeros has a cluster point.
	This implies, in view of the analyticity of $Q$, that $Q(s) = 0$ for all $s$ which contradicts the definition of $Q$.
	Hence, $Q$ has only a finite number of zeros in $\Omega$. This shows that $Q$ has only a finite number of zeros in the 
	domain $\mathcal{C}$ if $c \ne 0$. \\
	To deal with the case $c=0$, we write 
	\begin{align*}
		Q(s) &= s^{\alpha_1}\left (s^{\alpha_2} - as^{\alpha_2-\alpha_1}-b   \right ) \\ 
		 &=s^{\alpha_1}P(s), 
	\end{align*}
	where $P(s) = s^{\alpha_2} - as^{\alpha_2-\alpha_1}-b.$
	By repeating the above arguments for $P$, the proof is complete.
	
	(iv) See \cite[Proposition 1, Part 3b]{Kaslik2018}. 
\end{proof}

\begin{corollary}\label{hq1} Assume that $a,b,c > 0$ and that one of conditions
	\begin{itemize}
	\item [(i)] $c(\rho^2_2 - \rho_1^2) < ab < c(\rho_2^2 + \rho_1^2)$, 
	\item [(ii)] $ab \leq c\left ( \rho_2 -\rho_1 \right )^2$,
	\end{itemize}
	is satisfied where $\rho_1,\rho_2$ are defined in \eqref{eq:def-rho}.
	Then, the function $Q$ defined in \eqref{eq:def-Q}
	has no purely imaginary zero.
\end{corollary}

\begin{proof} 
	Consider the system \eqref{eq}. Due to the fact that $\rho_2 \neq 0$, this system is equivalent to 
	\begin{align}\label{eq1} 
		\left \{
		\begin{array}{rl}
			\omega^{\alpha_1} & = \displaystyle \frac{a}{\rho_2} + c\frac{\rho_1}{\rho_2}\omega^{-\alpha_2},
				\smallskip \\
	 		\rho_1 a \omega^{\alpha_2} + \rho_2 b \omega^{\alpha_1} &= c(\rho_2^2 - \rho_1^2).
		\end{array}
		\right.
	\end{align}
	Thus, we obtain
	\begin{equation}\label{add_eq}
		 a\rho_1\omega^{\alpha_2} + bc\rho_1\omega^{-\alpha_2}+ab - c(\rho_2^2-\rho_1^2) = 0. 
	\end{equation}
	Setting $X=\omega^{\alpha_2}$, equation \eqref{add_eq} takes the form
	\begin{align}\label{eeq}
		 a \rho_1 X^2 + \left[ ab - c (\rho_2^2-\rho_1^2) \right] X + b c \rho_1 = 0. 
	\end{align}
	The discriminant of the quadratic equation \eqref{eeq} is 
	\begin{align*}
	 	\Delta 
	 	& = \left ( ab - c(\rho_2^2-\rho_1^2) \right )^2 - 4abc\rho_1^2 \\
	 	& =a^2b^2 + c^2(\rho_2^2-\rho_1^2)^2-2abc(\rho_1^2 + \rho_2^2) \\
		& = \left(ab -c(\rho_1^2 + \rho_2^2)\right)^2 - 4c^2\rho_1^2\rho_2^2 \\
		& =\left ( ab - c(\rho_1+\rho_2)^2 \right )\left ( ab - c(\rho_2-\rho_1)^2 \right ).
	\end{align*}
	\noindent (i) Clearly, if $c(\rho^2_2 - \rho_1^2) < ab < c(\rho_2^2 + \rho_1^2),$ then $\Delta < 0.$ 
	Hence, the quadratic equation \eqref{eeq} has no real roots. This implies that
	the system \eqref{eq} has no root $\omega>0$. This together with Lemma \ref{bd3.1}(ii) and 
	Lemma \ref{bd3.1}(iv) shows that $Q$ has no purely imaginary zero.
	
	\noindent (ii) If $ab \leq c\left ( \rho_2 -\rho_1 \right )^2$, 
	the quadratic equation \eqref{eeq} has two (not necessarily distinct) real roots. Because $ 0 < \rho_1 < \rho_2,$ 
	we have $( \rho_2 -\rho_1  )^2 < \rho_2^2-\rho_1^2.$ This implies that $ab - c(\rho_2^2-\rho_1^2) < 0$. 
	Moreover, $a,b,c, \rho_1 > 0,$  
	thus the two roots of the quadratic equation \eqref{eeq} are negative. 
	Hence, in view of the relation $X = \omega^{\alpha_2}$ with $0 < \alpha_2 \le 1$
	between the solution $X$ of \eqref{eeq} and the solution $\omega$ of \eqref{eq},
	the system \eqref{eq} has no root $\omega>0$. Using Lemma 
	\ref{bd3.1}(ii) and Lemma \ref{bd3.1}(iv), we see that $Q$ has no purely imaginary zero. 
\end{proof}

Recall that if $c \leq 0,$ then $Q$ has at least one non-negative real zero, which precludes any 
kind of stability. 
Thus, in this section, we only consider the case $c >0$. As shown above, because $Q$ 
has only a finite number of zeros in the domain $\mathbb{C}$, 
there exists a constant $R > 0$ which is large enough such that $Q$ has no zero in
$\left \{z\in \mathbb C: |z|\geq R\right \}.$ On the other hand, $Q$ is continuous at $0$ with $Q(0) > 0$,
so we can find a small constant $\varepsilon > 0$ such that $Q(z)\neq 0$ 
in $\left \{z\in\mathbb C: |z| \leq \varepsilon \right \}$. 
We define an oriented contour $\gamma$ formed by four segments:
\begin{align*}
	\gamma_1 &:=\left \{s= i\omega: \varepsilon \leq \omega \leq R  \right \}; \\
	\gamma_2 &:=\left \{ s=R e^{i\varphi}: -\frac{\pi}{2}\leq \varphi  \leq \frac{\pi}{2} \right\}; \\
	\gamma_3 &:=\left \{ s=\varepsilon e^{i\varphi}: -\frac{\pi}{2}\leq \varphi  \leq \frac{\pi}{2} \right\}; \\
	\gamma_4 &:=\left \{s= i\omega: -R \leq \omega \leq -\varepsilon  \right \}.
\end{align*}
Clearly, if $Q$ has no purely imaginary zero, then all zeros in the
closed right hand side of the complex plane $\{s=r(\cos \phi+i\sin \phi)\in \mathbb C: r\geq 0, \phi\in (-\pi,\pi]\}$ 
of $Q$ (if they exist) must lie inside the contour $\gamma$. Based on Cauchy's argument principle in complex analysis, 
we see that $n(Q(C), 0) = Z - P,$ where $n(Q(C), 0)$ is the number of encirclements in the positive direction 
(counter-clockwise) around the origin of the the Nyquist plot $Q(\gamma)$, $Z$ and $P$ are the number of zeros and 
number of poles of $Q$ inside the contour $\gamma$ in the $s$-plane, respectively. Due to the fact that $Q$ is analytic 
inside $\gamma$, we have $P=0$ and thus $n(Q(C), 0) = Z.$ 
This implies that if $Q$ has no purely imaginary zero, then all roots of the equation 
$Q(s) = 0$ lie in the open left-half complex plane if and only if
$n(Q(C), 0) = 0.$ Notice that $Q(0)>0$, $\lim_{|s|\to\infty}|Q(s)|=\infty$, $\Re Q(i\omega) = \Re Q(-i\omega)$ and  $\Im Q(i\omega) = -\Im Q(-i\omega)$. It is easy to see that $n(Q(C), 0) = 0$ if $\Re Q(i\omega) > 0$ for any $\omega>0$ that satisfies $\Im Q(i\omega) = 0$.
Consider $\omega>0$ and put
\begin{align}
	 h_1(\omega) 
	 &:= \Re (Q(i\omega)) 
		 = \omega^{\alpha_1 + \alpha_2}\cos\frac{(\alpha_1 + \alpha_2)\pi}{2} 
	 		- a\omega^{ \alpha_2}\cos\frac{ \alpha_2\pi}{2} 
		 	- b\omega^{\alpha_1}\cos\frac{\alpha_1 \pi}{2}+c,  \\
	h_2(\omega) 
	&:= \Im (Q(i\omega)) 
		= \omega^{\alpha_1 + \alpha_2}\sin\frac{(\alpha_1 + \alpha_2)\pi}{2} 
			- a\omega^{ \alpha_2}\sin\frac{\alpha_2\pi}{2} 
			- b\omega^{\alpha_1} \sin\frac{\alpha_1 \pi}{2}.
\end{align}
If there exists some $\omega > 0$ such that $h_2(\omega) = 0,$ then 
\begin{align*}
	\sin\frac{(\alpha_1 + \alpha_2)\pi}{2}h_1(\omega) 
	&= \sin\frac{(\alpha_1 + \alpha_2)\pi}{2}h_1(\omega) -\cos\frac{(\alpha_1 + \alpha_2)\pi}{2}h_2(\omega) \\
	&= a\omega^{ \alpha_2}\left ( \sin\frac{ \alpha_2\pi}{2}\cos\frac{(\alpha_1 + \alpha_2)\pi}{2}
		- \cos\frac{ \alpha_2\pi}{2}\sin\frac{(\alpha_1 + \alpha_2)\pi}{2} \right ) \\
	& {} \quad + b\omega^{\alpha_1}\left ( \sin\frac{ \alpha_1\pi}{2}\cos\frac{(\alpha_1 + \alpha_2)\pi}{2}
		- \cos\frac{ \alpha_1\pi}{2}\sin\frac{(\alpha_1 + \alpha_2)\pi}{2} \right ) \\ 
	& {} \quad + c\sin\frac{(\alpha_1 + \alpha_2)\pi}{2} \\
	&= c\sin\frac{(\alpha_1 + \alpha_2)\pi}{2}-a\omega^{ \alpha_2}\sin\frac{ \alpha_1\pi}{2} 
		- b\omega^{\alpha_1}\sin\frac{ \alpha_2\pi}{2}.
\end{align*}
Thus, the variable $\omega>0$ satisfies the system
\begin{align}\label{qh1}
	\left\{
	\begin{array}{rl}
		 h_2(\omega) &= 0 \\
		 h_1(\omega) &= c - a\omega^{\alpha_2}q_1 - b\omega^{\alpha_1}q_2
	\end{array}
	\right.
\end{align}
with
\begin{align}
 	q_1 = \frac{\sin\frac{\alpha_1 \pi}{2}}{\sin\frac{(\alpha_2 + \alpha_1)\pi}{2}},\hspace{2cm}  
	q_2 = \frac{\sin\frac{\alpha_2 \pi}{2}}{\sin\frac{(\alpha_2 + \alpha_1)\pi}{2}}.
\end{align}
It is then clear from our assumptions on $\alpha_1$ and $\alpha_2$ that $q_1, q_2 > 0$.
Based on the analysis above, we obtain some sufficient conditions that ensure that the 
function $Q(\cdot)$ has no zero lying in the closed right half of the complex plane. 

\begin{lemma}\label{bd3.2}  
	Assume that $a, b \leq 0,$ and $c > 0$.  
	Then, all zeros of $Q$ are in the open left-half complex plane regardless of $\alpha_1$ and $\alpha_2$.
\end{lemma}

\begin{proof}
	See \cite[Proposition 1(b)]{Kaslik2018}.  
\end{proof}

\begin{lemma}\label{bd3.3} 
	Let $0 < \alpha_1 < \alpha_2 \leq 1$. Assume that $a = 0$, $b > 0$ and 
	\begin{equation}
		\label{eq:bd3.3}
		c > \left ( b q_1 \right )^{\alpha_1/\alpha_2} b q_2. 
	\end{equation}
	Then, all zeros of $Q$ lie in the open left-half of the complex plane.
\end{lemma}

\begin{proof} 
	Because $a = 0$ and $b > 0$, we see that $h_2(\omega_0) = 0$ if and only if 
	$\omega_0 = \left ( bq_1 \right )^{1/\alpha_2}$. From \eqref{qh1}, we have
	$h_1(\omega_0) = c - \left ( bq_1 \right )^{\alpha_1/\alpha_2} b q_2$. 
	By the assumption \eqref{eq:bd3.3}, we obtain $h_1(\omega_0) > 0$. 
	This implies that all zeros of $Q$ lie in the open left-half of complex plane.    
\end{proof}

\begin{lemma}\label{bd3.4} 
	Let $0 < \alpha_1 < \alpha_2 \leq 1.$ Assume that $b = 0$, $a > 0$ and  
	\begin{equation}
		\label{eq:bd3.4}
		c > \left ( a q_2 \right )^{\alpha_2/\alpha_1} a q_1.
	\end{equation}
	Then, all zeros of $Q$ lie in the open left-half complex plane. 
\end{lemma}

\begin{proof}
	Since $b = 0$ and $a > 0$, it is easy to show that $h_2(\omega_0) = 0$ if and only if 
	$\omega_0 = \left ( aq_2 \right )^{1/\alpha_1}$. From \eqref{qh1}, it follows that
	$h_1(\omega_0) = c - \left ( aq_2 \right )^{\alpha_2/\alpha_1}aq_1$. 
	By the assumption \eqref{eq:bd3.4},
	we see that $h_1(\omega_0) > 0$
	which implies that all zeros of $Q$ lie in the open left-half of the complex plane.
\end{proof}

\begin{lemma}\label{bd3.5} 
	Let $0 < \alpha_1 < \alpha_2 \leq 1.$ Assume that $a, b, c > 0 $. 
	Then, all zeros of $Q$ are in the open left-half of the complex plane if one of the following conditions holds:
	\begin{itemize}
	\item [(i)]  $  aq_2+bq_1 > 1$ and $aq_2\left( (a+b)q_2  \right )^{\alpha_2/\alpha_1} + b(a+b)q_2^2  \leq c.$
	\item [(ii)] $  aq_2+bq_1 \leq 1$ and $ aq_1 + bq_2 < c.$
	\end{itemize} 
\end{lemma}

\begin{proof} 
	We have 
	\begin{align}\label{g2}
		 h_2'(\omega)  
		 &= (\alpha_1 + \alpha_2)\omega^{\alpha_1 + \alpha_2-1}\sin\frac{(\alpha_1 + \alpha_2)\pi}{2} 
		 	- a\alpha_2\omega^{ \alpha_2-1}\sin\frac{ \alpha_2\pi}{2}
			- \alpha_1\omega^{\alpha_1-1 }\sin\frac{\alpha_1 \pi}{2}  \nonumber\\
		&= \omega^{\alpha_1-1}\left ( (\alpha_1 + \alpha_2)\omega^{ \alpha_2}\sin\frac{(\alpha_1 + \alpha_2)\pi}{2} 
				- a\alpha_2\omega^{ \alpha_2-\alpha_1}\sin\frac{ \alpha_2\pi}{2}
				- b\alpha_1 \sin\frac{\alpha_1 \pi}{2} \right )\nonumber \\
		&= \omega^{\alpha_1-1}g_2(\omega)
	\end{align}
	where 
	\begin{equation}
		\label{eq:def-g2}
		g_2(\omega)
			: =  
			(\alpha_1 + \alpha_2)\omega^{ \alpha_2}\sin\frac{(\alpha_1 + \alpha_2)\pi}{2} 
			- a \alpha_2\omega^{ \alpha_2-\alpha_1}\sin\frac{ \alpha_2\pi}{2}
			- b \alpha_1 \sin\frac{\alpha_1 \pi}{2}.
	\end{equation}
	Notice that
	\begin{align*}
		 g_2'(\omega)
		 &= \alpha_2(\alpha_1 + \alpha_2) \omega^{\alpha_2 - 1}\sin\frac{(\alpha_1 + \alpha_2)\pi}{2}
			- (\alpha_2 - \alpha_1)\alpha_2 a\omega^{\alpha_2 - \alpha_1 -1}\sin\frac{ \alpha_2\pi}{2}. 
	\end{align*}
	It is not difficult to check that $g_2'(\omega) < 0$ in $(0,\omega_1)$ 
	and $g_2'(\omega) > 0$ in $(\omega_1,\infty),$ where
	$\omega_1 = \left ( \frac{\alpha_2-\alpha_1}{\alpha_1 + \alpha_2}aq_2 \right )^{1/\alpha_1}.$ 
	Due to the fact that $g_2(0) = -\alpha_1 b \sin\frac{\alpha_1\pi}{2} < 0,$ and 
	$\lim_{\omega \to +\infty}g_2(\omega) = +\infty,$ the equation $g_2(\omega) = 0$ 
	has a unique root $\omega_2  \in (0, \infty).$ Moreover $g_2(\omega) < 0$ in $(0,\omega_2)$ and 
	$g_2(\omega) > 0$ in $(\omega_2, \infty).$
	Hence, $h_2$ is decreasing in $(0,\omega_2)$
	and increasing in $(\omega_2, \infty).$ On the other hand, $h_2(0) = 0$ and  
	$\lim_{\omega \to +\infty}h_2(\omega) = +\infty$. This shows that the equation $h_2(\omega) = 0$ 
	has a unique root $\omega_3 \in (0,\infty)$ and then $h_2(\omega) < 0$ for all $\omega \in (0,\omega_3)$
	and $h_2(\omega) > 0$ for all $\omega \in (\omega_3,\infty).$ 
	
	\noindent (i) If $  aq_2+bq_1 > 1$, then $h_2(1) < 0.$ This implies that $\omega_3 > 1.$ 
	Moreover, due to $\alpha_1 < \alpha_2 \leq 1$, we have  
	\begin{align}
		 h_2(\omega) 
		 > \omega^{\alpha_1 + \alpha_2}\sin\frac{(\alpha_1 + \alpha_2)\pi}{2}
		 	- (a+b)\omega^{\alpha_2}\sin\frac{\alpha_2\pi}{2}
	\end{align}
	for every $\omega > 1$,
	and thus $h_2\left( \left( (a+b)q_2  \right )^{1/\alpha_1}  \right) > 0.$ 
	This implies that $1 < \omega_3 < \left( (a+b)q_2  \right )^{1/\alpha_1}.$ 
	Hence, if 
	\begin{align*}
		c \geq aq_2\left( (a+b)q_2  \right )^{\alpha_2/\alpha_1} + b(a+b)q_2^2,  
	\end{align*}
	due to $q_2 > q_1 > 0,$ we obtain $c > a\omega_3^{\alpha_2}q_1 + b\omega_3^{\alpha_1}q_2,$ 
	which together with \eqref{qh1} leads to  $h_1(\omega_3) > 0.$ The proof of this part is complete.
	
	\noindent (ii) If $  aq_2+bq_1 \leq 1$, then $h_2(1) \geq 0.$ 
	This implies that $0 < \omega_3 \leq 1.$ Due to $ c > aq_1 + bq_2,$ we have
	$c > a\omega_3^{\alpha_2}q_1 + b\omega_3^{\alpha_1}q_2,$. This together with \eqref{qh1} shows 
	that $h_1(\omega_3) > 0.$ The proof is finished.
\end{proof}

\begin{lemma}\label{bd3.6} 
	Assume that $a < 0$ and $b, c > 0$. 
	Then, all zeros of $Q$ are in the open left-half complex plane if one of the following conditions holds:
	\begin{itemize}
	\item [(i)]  $a q_2+b q_1 > 1$ and $\left ( b q_1 \right )^{\alpha_1/\alpha_2} b q_2 \leq c.$
	\item [(ii)] $a q_2+b q_1 \leq 1$ and $b q_2 \leq c$.
	\end{itemize} 
\end{lemma}

\begin{proof}
	As shown in the proof of Lemma \ref{bd3.5}, 
	we have $h_2'(\omega) = \omega^{\alpha_1-1} g_2(\omega)$
	where $g_2$ is as in \eqref{eq:def-g2}.
	Notice that
	\[
		g_2'(\omega)
		= \alpha_2(\alpha_1 + \alpha_2) \omega^{\alpha_2 - 1}\sin\frac{(\alpha_1 + \alpha_2)\pi}{2}
			- (\alpha_2 - \alpha_1)\alpha_2 a\omega^{\alpha_2 - \alpha_1 -1}\sin\frac{ \alpha_2\pi}{2}
		> 0
	\]
	for $\omega \in (0,\infty)$. Due to the facts that $g_2(0) = -\alpha_1 b \sin\frac{\alpha_1\pi}{2} < 0$ and 
	$\lim_{\omega \to +\infty}g_2(\omega) = +\infty,$ 
	the equation $g_2(\omega) = 0$ has a unique root $\omega_1  \in (0, \infty)$. 
	Moreover $g_2(\omega) < 0$ in $(0,\omega_1)$
	and $g_2(\omega) > 0$ in $(\omega_1, \infty).$
	This shows that $h_2$ is decreasing on $(0,\omega_1)$
	and increasing on $(\omega_1, \infty)$. 
	On the other hand, since $h_2(0) = 0$ and  $\lim_{\omega \to +\infty}h_2(\omega) = +\infty$, 
	the equation $h_2(\omega) = 0$ 
	has a unique root $\omega_2 \in (0,\infty)$ and $h_2(\omega) < 0$ for all $\omega \in (0,\omega_2)$
	and $h_2(\omega) > 0$ for all $\omega \in (\omega_2,\infty).$ 
	 
	\noindent (i) If $aq_2+bq_1 > 1$ then $h_2(1) < 0.$ Thus $\omega_2 > 1$. 
	Moreover, since $a < 0,$  we have  
	\[
		h_2(\omega) 
		> \omega^{\alpha_1 + \alpha_2}\sin\frac{(\alpha_1 + \alpha_2)\pi}{2}
			- b\omega^{\alpha_1}\sin\frac{\alpha_1\pi}{2},
	\]
	and thus $h_2\left( (bq_1)^{1/{\alpha_2}} \right) > 0$. This implies that 
	$1 < \omega_2 < (bq_1)^{1/{\alpha_2}}.$ 
	From that if 
	\begin{align*}
		\left ( bq_1 \right )^{\alpha_1/\alpha_2}bq_2 \leq c,  
	\end{align*}
	we obtain $c > a\omega_2^{\alpha_2}q_1 + b\omega_2^{\alpha_1}q_2,$ 
	which together with \eqref{qh1} leads to  $h_1(\omega_2) > 0.$
	
	\noindent (ii) If $  aq_2+bq_1 \leq 1$, then $h_2(1) \geq 0.$ Thus $0 < \omega_2 \leq 1.$ 
	Due to $ c \geq  bq_2$, we see that
	$c > a \omega_2^{\alpha_2}q_1 + b\omega_2^{\alpha_1}q_2$. The proof is completed.
\end{proof}

\section{Estimates for the functions $\mathcal R^\lambda$ and $\mathcal S^\beta$}
\label{sec:specfun}

This section is devoted to the study of some important estimates of the functions 
$\mathcal R^\lambda$ and $\mathcal S^\beta$ on $(0,\infty)$. We recall their definitions from \eqref{eq:def-r-s}, viz.\ 
\begin{align*}
	\mathcal R^\lambda(t) 
	&= \mathcal L^{-1}\left \{\frac{s^{l(\alpha)-\lambda}}{sQ(s)}  \right \}(t) , 
		\quad \lambda \in \left \{ 0,\alpha_1, \alpha_2 \right \}, \\ 
 	\mathcal S^\beta(t)
 	&= \mathcal L^{-1}\left \{\frac{s^{l(\alpha)-\beta}}{Q(s)}  \right \}(t) , 
 		\quad \beta \in \left \{ \alpha_1, \alpha_2 ,l(\alpha)\right \},  
\end{align*}
where $l(\alpha) = \alpha_1 + \alpha_2$ and $Q(s) = s^{\alpha_1 + \alpha_2} -a_{11}s^{\alpha_2} - a_{22}s^{\alpha_1} + \det A$.  

\begin{lemma}\label{dl4.2}
	Let $\alpha_1, \alpha_2 \in (0, 1]$ and denote $  \nu = \min\{\alpha_1, \alpha_2 \}$. 
	Assume there are no zeros of the characteristic function $Q$ in the closed right-half complex plane. 
	Then, the following estimates hold for 
	$\lambda \in \left \{ 0,\alpha_1, \alpha_2 \right \}$ and $\beta \in \left \{ \alpha_1, \alpha_2 ,l(\alpha)\right \}$:
	\begin{align}
		\label{R1}
		\mathcal R^\lambda(t) &= O(t^{-\nu}) \quad \mbox{ as } t \to \infty , \\
		\label{S1}
		\mathcal S^\beta(t) &= O(t^{-\nu-1})  \quad \mbox{ as } t \to \infty , \\
		\label{S2}
		\mathcal S^\beta(t) &= O(t^{\nu-1})  \quad \mbox{ as } t \to 0 .
	\end{align}
	Moreover, 
	\begin{equation}\label{S3}
		\int_{0}^\infty |\mathcal S^\beta(t)| dt < \infty.
	\end{equation}
\end{lemma}

The proof of the lemma is quite lengthy and technical. Therefore, in order not to distract the reader and to make it easier  
to focus on the main results, we provide the proof in the Appendix at the end of the paper. 

Our first application of Lemma \ref{dl4.2} deals with estimates for the convolution of $\mathcal S^\beta$ and a continuous function. 

\begin{theorem}\label{dl2}
	Let $\alpha_1, \alpha_2 \in (0,1]$ and $\beta \in \left \{ \alpha_1,\alpha_2,l(\alpha) \right \}$.
	For each continuous function $g : [0, \infty) \rightarrow \mathbb R$, we put
	\begin{align}
		 F_g^\beta (t) := \mathcal S^\beta \ast g(t) 
		 	= \int_0^t\mathcal S^\beta(t-s)g(s)ds
	\end{align}
	and
	\begin{align}
		\bar  F_g^\beta (t) := |\mathcal S^\beta| \ast |g| (t) 
		 	= \int_0^t\mathcal |S^\beta(t-s)| \cdot | g(s) |  ds.
	\end{align}
	Assume that all zeros of the characteristic function 
	$Q$ are in the open left-half complex plane. Then, the following statements hold. 
	\begin{itemize}
	\item [(i)] If $g$ is bounded then $\bar F_g^\beta$ and $F_g^\beta$ are also bounded.
	\item [(ii)] If $\lim_{t \to \infty} g(t) = 0$ then $\lim_{t \to \infty} \bar F_g^\beta(t) = \lim_{t \to \infty} F_g^\beta(t) = 0$.
	\item [(iii)] If there exists some $\eta \ge 0$ such that $g(t) = O(t^{-\eta})$ for $t \to \infty$ then
	 	\[
	 		\bar F_g^\beta(t) = O(t^{-\mu}) 
	 		\mbox{ and }
	 		F_g^\beta(t) = O(t^{-\mu}) 
	 		\mbox{ for } t \to \infty
		\]
	 	where $\mu = \min \left\{\alpha_1, \alpha_2, \eta \right\}.$ 
	\end{itemize}
\end{theorem}

\begin{proof} 
	We denote $\nu = \min\{\alpha_1, \alpha_2 \}$. Since, by definition, $|F_g^\beta(t)| \le \bar F_g^\beta(t)$,
	the claims for $F_g^\beta$ immediately follow from those for $\bar F_g^\beta$, and therefore it suffices
	to explicitly prove the latter.

	Statement (i) is merely the special case of $\eta = 0$ of part (iii).

	
	To prove (ii), we note that $\bar F_g^\beta(t) \ge 0$ by definition. Therefore, it is sufficient to show that
	for every $\varepsilon > 0$ there exists a constant $\widetilde T = \widetilde T(\varepsilon)$ such that
	\begin{equation}
		\label{eq:t42iia}
		\bar F_g^\beta(t) \le \varepsilon \quad \mbox{ for all } t > \widetilde T.
	\end{equation}
	Since this is trivially fulfilled if $g(t) = 0$ for all $t$, we from now on assume that $g(t) \ne 0$ for some $t$, and 
	hence $\| g \|_\infty > 0$. 
	
	Our first observation is then that, from \eqref{S1} and \eqref{S2}, we know that there exists some constant 
	$C > 0$ such that
	\begin{equation}
		\label{eq:t42iib}
		|S^\beta(t)| 
			\le 
			\begin{cases}
				C t^{-\nu-1} & \mbox{ for } t \ge 1, \\
				C t^{\nu-1}  & \mbox{ for } t \le 1.
			\end{cases}
	\end{equation}
	Given an arbitrary $\varepsilon > 0$, due to our assumption on $g$ we may then find some $\hat T > 0$
	such that $|g(t)| < \nu \varepsilon / (3 C)$ for all $t > \hat T$. Using these values $\hat T$ and $C$, we then define
	\[
		\widetilde T = \hat T + \max \left \{ 1, \left( \frac{3 C \| g \|_\infty \hat T}{\varepsilon}\right)^{1/(\nu+1)} \right \}. 
	\]
	For $t > \widetilde T \ge \hat T + 1$, we can then write
	\begin{align}
		\label{31}
		\bar F_g^\beta(t) 
		& = \int_0^{\hat T} |\mathcal S^\beta(t-s)| \cdot |g(s)|ds 
			+ \! \int_{\hat T}^{t-1}  \! |\mathcal S^\beta(t-s)| \cdot |g(s)|ds 
			+ \! \int_{t-1}^{t} \! |\mathcal S^\beta(t-s)| \cdot |g(s)|ds \nonumber \\
		& = F_1(t) + F_2(t) +F_3(t).
	\end{align}
	Our goal now is to show that, under these assumptions, $F_j(t) \le \varepsilon/3$ for $j = 1, 2, 3$, which implies 
	\eqref{eq:t42iia} and thus suffices to prove part (ii) of the Theorem. In this context, we see that
	\begin{align*}
		F_1(t) 
		& = \int_0^{\hat T} |\mathcal S^\beta(t-s)| \cdot |g(s)|ds 
		\le \| g \|_\infty C \int_0^{\hat T} (t - s)^{-\nu-1} d s \\
		& \le  \| g \|_\infty C \hat T (t - \hat T)^{-\nu-1} < \frac \varepsilon 3.
	\end{align*}
	because here $t - s \ge t - \hat T > \widetilde T - \hat T \ge 1$, so that we may use the first of the bounds given in 
	\eqref{eq:t42iib}. In the penultimate step, we have bounded the integral by the product of the length of the integration
	interval and the maximum of the integrand, and in the last step, we have used the fact that $t > \widetilde T$ and the
	definition of $\widetilde T$.
	
	Furthermore,
	\begin{align*}
		F_2(t) 
		& = \int_{\hat T}^{t-1} |\mathcal S^\beta(t-s)| \cdot |g(s)|ds 
		\le \frac{\nu \varepsilon}{3 C} C \int_{\hat T}^{t-1} (t - s)^{-\nu-1} d s \\
		& = \frac{\nu \varepsilon}{3} \frac 1 \nu \left(  1 - (t - \hat T)^{-\nu} \right)
		< \frac \varepsilon 3
	\end{align*}
	because here $s \ge \hat T$, so that $|g(s)| \le \nu \varepsilon / (3 C)$, and $t-s \ge 1$, so we may once again use 
	the first bound of \eqref{eq:t42iib}.
	
	Finally, 
	\begin{align*}
		F_3(t) 
		& = \int_{t-1}^t |\mathcal S^\beta(t-s)| \cdot |g(s)|ds 
		\le \frac{\nu \varepsilon}{3 C} C \int_{t-1}^t (t - s)^{\nu-1} d s 
		= \frac{\nu \varepsilon}{3} \frac 1 \nu 
		= \frac \varepsilon 3
	\end{align*}
	where now $t$ and $s$ are such that we may invoke the second bound of \eqref{eq:t42iib} but, as in the
	previous step, $s \ge \hat T$, so that once again $|g(s)| \le \nu \varepsilon / (3 C)$.
	This completes the proof of part (ii) of the Theorem.

	For the proof of (iii), we note that \eqref{eq:t42iib} is valid in this case too. Moreover, since we are interested in the asymptotic 
	behaviour of $\bar F_g^\beta(t)$ for large $t$, we may assume without loss of generality that $t \ge 2$. Then we write
	\begin{align}\label{35}
		\bar F_g^\beta(t) 
		& = \int_0^{1}|\mathcal S^\beta(t-s)| \cdot |g(s)|ds +\int_1^{t/2} |\mathcal S^\beta(t-s)| \cdot |g(s)| ds \nonumber \\ 
		& \phantom{ = } {} \quad + \int_{t/2}^{t-1} |\mathcal S^\beta(t-s)| \cdot |g(s)|ds 
					+ \int_{t-1}^{t}|\mathcal S^\beta(t-s)| \cdot |g(s)|ds \nonumber \\
		& = F_4(t) + F_5(t) + F_6(t) + F_7(t),
	\end{align}
	and we need to show that $F_j(t) = O(t^{-\mu})$ for $j = 4, 5, 6, 7$.
	
	In this connection, we first note that, by assumption, 
	\begin{equation}
		\label{f2}
		|g(t)| \leq C' t^{-\eta} \quad \forall t \geq 1
	\end{equation}
	with some $C' > 0$, so that the upper branch of \eqref{eq:t42iib} implies
	\begin{align*}
		0 \le 
		F_4(t) 
		& \le C \| g \|_\infty \int_0^1 (t-s)^{-\nu-1} ds 
		= \frac{C  \| g \|_\infty}{\nu} \left( (t-1)^{-\nu} - t^{-\nu} \right) \\
		& <  \frac{C  \| g \|_\infty}{\nu} (t-1)^{-\nu}
		= O(t^{-\nu}) 
		= O(t^{-\mu})
	\end{align*}
	and
	\begin{align*}
		0 \le 
		F_5(t) 
		& \le C C' \int_1^{t/2} (t-s)^{-\nu-1} s^{-\eta} ds 
		\le C C' \int_1^{t/2} (t-s)^{-\nu-1} ds  \\
		& = \frac{C C'}{\nu} \left( (t-1)^{-\nu} - (t/2)^{-\nu} \right) 
		< \frac{C C'}{\nu} (t-1)^{-\nu} 
		= O(t^{-\nu}) 
		= O(t^{-\mu})
	\end{align*}
	as well as
	\begin{align*}
		0 \le 
		F_6(t) 
		& \le C C' \int_{t/2}^{t-1} (t-s)^{-\nu-1} s^{-\eta} ds 
		\le C C' \left( \frac t 2 \right)^{-\eta} \int_{t/2}^{t-1} (t-s)^{-\nu-1} ds  \\
		& = 2^\eta \frac{C C'}{\nu} t^{-\eta} \left( 1 - (t/2)^{-\nu} \right)
		< 2^\eta \frac{C C'}{\nu} t^{-\eta}
		= O(t^{-\eta}) 
		= O(t^{-\mu}).
	\end{align*}
	For the remaining part we need to invoke the second branch of \eqref{eq:t42iib} in combination with \eqref{f2}
	to derive
	\begin{align*}
		0 \le 
		F_7(t) 
		& \le C C' \int_{t-1}^t (t-s)^{\nu-1} s^{-\eta} ds 
		\le C C' (t-1)^{-\eta} \int_{t-1}^t (t-s)^{\nu-1} ds  \\
		& = \frac{C C'}{\nu} (t-1)^{-\eta} 
		= O(t^{-\eta}) 
		= O(t^{-\mu}),
	\end{align*}
	thus completing the proof.
\end{proof}

As an immediate application of Theorem \ref{dl2}(iii), we can conclude that 
\begin{equation}
	\label{eq:rm2a}
	\bar F^\beta_{g_1} (t) = O(t^{-\nu})
	\quad \mbox{ as } t \to \infty
	\qquad
	\mbox{ for } 
	g_1(t) = \min \{ 1, t^{-\nu} \}.
\end{equation}
Moreover, assuming $\nu < 1$ and setting
\[
	g_2(t) = t^{-\nu} - g_1(t) 
	= \begin{cases}
		t^{-\nu} - 1 & \mbox{ for } t \in [0,1], \\
		0                  & \mbox{ for } t > 1,
	\end{cases}
\]
we can obtain (using Lemma \ref{dl4.2} and the classical relation between the 
incomplete Beta function and the hypergeometric Function $_2 F_1$,
cf.\ \cite[eq.\ (6.6.8)]{AS}) the following bounds:
\begin{itemize}
\item If $t \ge 2$ then we have
	\begin{align}
		\nonumber
		\bar F^\beta_{g_2}(t) 
		& = \int_0^1 | \mathcal S^\beta (t-s) | (s^{-\nu}-1) ds 
		\le  C \int_0^1 (t-s)^{-\nu-1} s^{-\nu} ds 
		= C t^{-2\nu} B_{1/t} (1-\nu, -\nu) \\
		& = \frac{C}{1-\nu} t^{-\nu-1} {} {}_2 F_1(1-\nu, 1+\nu; 2-\nu; t^{-1})
		\le C' t^{-\nu-1}
		\label{eq:rm2b}
	\end{align}
	with some $C' > 0$.
\item If $t \in [1,2]$ then 
	\begin{align}
		\nonumber
		\bar F^\beta_{g_2}(t) 
		& = \int_0^1 | \mathcal S^\beta (t-s) | (s^{-\nu}-1) ds 
		\le  C \int_0^1 (t-s)^{\nu-1} s^{-\nu} ds 
		= C B_{1/t} (1-\nu, \nu) \\
		& = \frac{C}{1-\nu} t^{\nu-1} {} {}_2 F_1(1-\nu, 1-\nu; 2-\nu; t^{-1})
		\le C'' t^{-\nu}
		\label{eq:rm2c}
	\end{align}
	with some $C'' > 0$.
\end{itemize}
Since $\bar F^\beta_{g_1}(t) + \bar F^\beta_{g_2}(t) = \bar F^\beta_{g_1 + g_2}(t)$,
we can summarize the observations of eqs.\ \eqref{eq:rm2a}, \eqref{eq:rm2b} and \eqref{eq:rm2c} 
in the following way:
\begin{remark}\label{rm2}
	Assuming that $\nu = \min \{ \alpha_1, \alpha_2\} < 1$,
	there exists a constant $C$ such that, for all $t \ge 1$ and $\beta \in \left \{ \alpha_1, \alpha_2, l(\alpha) \right \}$,
	\begin{align}
		 t^\nu \int_0^t |\mathcal S^\beta(t-s)| \frac{1}{s^\nu}ds  \leq C.
	\end{align}
\end{remark}

\section{Asymptotic behaviour of solutions\newline to non-commensurate fractional planar systems}
\label{sec:mainresults}

In this section we will study the asymptotic behaviour of solutions to fractional-order linear planar systems
and the Mittag-Leffler stability of an equilibrium point to fractional nonlinear planar systems.

\subsection{Asymptotic behaviour of solutions to fractional linear planar systems}

Consider the  non-homogeneous linear two-component incommensurate fractional-order system
\begin {subequations}
\label{eq:system1}
\begin{align}\label{4.1.1}
	 ^C D^\alpha_{0^+} x(t)& = Ax(t) + f(t),\; t>0,\\
	 x(0) & = x^0 \in \mathbb R^2,
\end{align}
\end{subequations}
where $\alpha=(\alpha_1,\alpha_2)\in (0,1]\times (0,1]$ is a multi index, 
$A = \left ( a_{ij} \right ) \in \mathbb R^{2 \times 2}$ is a square real matrix and $f = (f_1, f_2)$ is a continuous vector valued function 
which is exponentially bounded on $[0,\infty)$. 

\begin{theorem}\label{dl3}
	Suppose that all zeros of the characteristic function 
	$Q=s^{\alpha_1+\alpha_2}-a_{11}s^{\alpha_2}-a_{22}s^{\alpha_1}+ \det A$
	lie in the open left-half of the complex plane. Then, the following statements hold. 
	\begin{itemize}
	\item [(i)] If $f$ is bounded, then for any $x^0\in\mathbb R^2$ the solution to \eqref{eq:system1} is also bounded.
	\item [(ii)] If $\lim_{t \to \infty} f(t) = 0$ 
		then the solution to \eqref{eq:system1} tends to $0$ when $t \to \infty$ for any $x^0\in\mathbb R^2$.
	\item [(iii)] If $\| f(t) \| = O(t^{-\eta})$ as $t \to \infty$ with some $\eta> 0$
		then every solution $x$ of \eqref{4.1.1} behaves as $\| x(t) \| = O(t^{-\mu})$ for $t \to \infty$ 
		where $\mu = \min \left\{ \alpha_1, \alpha_2, \eta \right\}.$
	\end{itemize}
\end{theorem}

\begin{proof} 
	The proof is straightforward by combining Lemma \ref{BTHS1}, Lemma \ref{dl4.2} and Theorem \ref{dl2}. 
\end{proof}

Based on Theorem \ref{dl3} and Lemmas \ref{bd3.2}, \ref{bd3.3}, \ref{bd3.4}, \ref{bd3.5} and \ref{bd3.6}, we obtain the following corollary.

\begin{corollary}
	Let 
	\[
			q_1= \frac{\sin \frac{\alpha_1\pi}{2}}{\sin\frac{(\alpha_1+\alpha_2)\pi}{2}} 
			\quad \mbox{ and } \quad
			q_2= \frac{\sin \frac{\alpha_2\pi}{2}}{\sin\frac{(\alpha_1+\alpha_2)\pi}{2}}.
	\]
	The statements of Theorem \ref{dl3} (i), (ii) and (iii) are true if one of the following conditions is satisfied.
	\begin{itemize}
	\item[(i)] $a_{11},a_{22}\leq 0$ and $\det A>0$.
	\item[(ii)] $a_{11}=0$, $a_{22}>0$, $\det A>0$ and 
		\[ 
			(a_{22}q_1)^{\alpha_1/\alpha_2}a_{22}q_2< \det A.
		\]
	\item[(iii)] $a_{22}=0$, $a_{11}>0$, $\det A>0$ and 
		\[ (a_{11}q_2)^{\alpha_2/\alpha_1}a_{11}q_1< \det A.\]
	\item [(iv)] $a_{11}, a_{22}, \det A > 0$ and one of the following conditions holds:
		\begin{itemize}
		\item [(iv)$_1$]  $a_{11}q_2+a_{22}q_1 > 1$ and 
			$a_{11}q_2\left( (a_{11}+a_{22})q_2  \right )^{\alpha_2/\alpha_1} + a_{22}(a_{11}+a_{22})q_2^2  \leq \det A$;
		\item [(iv)$_2$] $  a_{11}q_2+a_{22}q_1 \leq 1$ and $ a_{11}q_1 + a_{22}q_2 < \det A.$
		\end{itemize} 
	\item[(v)] $a_{11} < 0$, $a_{22}, \det A > 0$  and one of the following conditions holds:
		\begin{itemize}
		\item [(v)$_1$]  $  a_{11}q_2+a_{22}q_1 > 1$ and $\left ( a_{22}q_1 \right )^{\alpha_1/\alpha_2}a_{22}q_2 \leq \det A$;
		\item [(v)$_2$] $  a_{11}q_2+a_{22}q_1 \leq 1$ and $  a_{22}q_2 \leq \det A$.
		\end{itemize} 
	\end{itemize}
\end{corollary}

\subsection{Mittag-Leffler stability of fractional nonlinear planar systems} \label{SST}

We now look at a different class of systems. Specifically, we now allow the differential equations
to contain nonlinearities, but we do require them to have the structure of an autonomous system,
i.e., we consider a fractional nonlinear planar system of the form
\begin{subequations}
\label{eq:4}
\begin{align}\label{4}
	 ^C D^\alpha_{0^+}x(t)& = Ax(t) +f(x(t)),\quad t > 0,\\
	  x(0)& = x^0 \in \Omega\subset\mathbb R^2,\label{ic4}
\end{align}
\end{subequations}
where $\alpha=(\alpha_1,\alpha_2)\in (0,1]\times (0,1]$ is a multi-index,  
$A = \left ( a_{ij} \right ) \in \mathbb R^{2 \times 2}$ is a square real matrix, 
$\Omega$ is an open subset of $\mathbb R^2$ containing the origin and 
$f:\Omega\rightarrow \mathbb R^2$ is locally Lipschitz continuous at the origin such that
$f(0)=0$ and $\lim_{r\to 0} l_f (r)=0$ with 
\[
	l_f(r) := \sup_{x,y\in B(0,r),\;x\ne y}\frac{\|f(x)-f(y)\|}{\|x-y\|}.
\]

\begin{definition}
	\label{def:mlstable}
	The trivial solution of \eqref{4} is Mittag-Leffler stable if there exist positive
	constants $\gamma,m$ and $\delta$ such that for any initial condition $x^0\in B(0,\delta)$, 
	the solution $\varphi(\cdot, x^0)$ of the initial value problem \eqref{eq:4} exists globally on the interval $[0,\infty)$ and
	\[
		\displaystyle\max\{\sup_{t\in [0,1]}\|\varphi(t, x^0)\|,\sup_{t \ge 1} t^\gamma\|\varphi(t, x^0)\|\}\le m.
	\]
\end{definition}

Our aim is to prove the following theorem.

\begin{theorem} \label{dl4.4} 
	Suppose that all zeros of the characteristic function 
	$Q(s) = s^{\alpha_1+\alpha_2}-a_{11}s^{\alpha_2}-a_{22}s^{\alpha_1}+ \det A$ 
	lie in the open left-half of the complex plane. Then, the trivial solution of differential equation \eqref{4} 
	is Mittag-Leffler stable. More precisely, there exist constants $\delta, \varepsilon > 0$  such that for 
	any $\|x^0\| < \delta$, the unique solution $\varphi(\cdot,x^0)$ of the initial value problem \eqref{eq:4} 
	exists globally on $[0,\infty)$ and $\sup_{t \geq 1} t^\nu \|\varphi(t,x^0)\| \leq \varepsilon$
	with $\nu = \min \{\alpha_1, \alpha_2\}.$
\end{theorem}

As shown above, we see that Lemmas \ref{bd3.2}, \ref{bd3.3}, \ref{bd3.4}, \ref{bd3.5} and \ref{bd3.6} give sufficient conditions which ensure that the characteristic function $Q$ has no zero in the closed right hand side of the complex plane. Thus, by combining these lemmas and Theorem \ref{dl4.4}, we obtain the result below.

\begin{corollary}
	Let 
	\[
			q_1= \frac{\sin \frac{\alpha_1\pi}{2}}{\sin\frac{(\alpha_1+\alpha_2)\pi}{2}} 
			\quad \mbox{ and } \quad
			q_2= \frac{\sin \frac{\alpha_2\pi}{2}}{\sin\frac{(\alpha_1+\alpha_2)\pi}{2}}.
	\]
	The statement of Theorem \ref{dl4.4} is true if one of the following conditions is satisfied.
	\begin{itemize}
	\item[(i)] $a_{11},a_{22}\leq 0$ and $\det A>0$.
	\item[(ii)] $a_{11}=0$, $a_{22}>0$, $\det A>0$ and 
		\[ 
			(a_{22}q_1)^{\alpha_1/\alpha_2}a_{22}q_2< \det A.
		\]
	\item[(iii)] $a_{22}=0$, $a_{11}>0$, $\det A>0$ and 
		\[ 
			(a_{11}q_2)^{\alpha_2/\alpha_1}a_{11}q_1< \det A.
		\]
	\item [(iv)] $a_{11}, a_{22}, det A > 0$ and one of the following conditions holds:
		\begin{itemize}
		\item [(iv)$_1$]  $a_{11}q_2+a_{22}q_1 > 1$ and 
			$a_{11}q_2\left( (a_{11}+a_{22})q_2  \right )^{\alpha_2/\alpha_1} + a_{22}(a_{11}+a_{22})q_2^2  \leq \det A$;
		\item [(iv)$_2$] $  a_{11}q_2+a_{22}q_1 \leq 1$ and $ a_{11}q_1 + a_{22}q_2 < \det A.$
		\end{itemize} 
	\item[(v)] $a_{11} < 0$, $a_{22}, det A > 0$  and one of the following conditions holds:
		\begin{itemize}
		\item [(v)$_1$]  $  a_{11}q_2+a_{22}q_1 > 1$ and $\left ( a_{22}q_1 \right )^{\alpha_1/\alpha_2}a_{22}q_2 \leq \det A$;
		\item [(v)$_2$] $  a_{11}q_2+a_{22}q_1 \leq 1$ and $  a_{22}q_2 \leq \det A$.
		\end{itemize} 
	\end{itemize}
\end{corollary}

\begin{proof}[Proof of Theorem \ref{dl4.4}]
	From the assumption of the theorem that $f$ is locally Lipschitz continuous at the origin, we can find a constant $\varepsilon_0 > 0$ 
	such that the function $f$ is Lipschitz continuous on $B(0,\varepsilon_0)$. Denote by $\hat{f}$ a Lipschitz extension of $f$ 
	to $\mathbb R^2$. This means that $\hat{f}$ is globally Lipschitz continuous and $\hat{f}(x)=f(x)$ on $B(0,\varepsilon_0)$. 
	We now focus on the system
	\begin{subequations}
	\label{eq:5}
	\begin{align}\label{5}
		 ^C D^\alpha_{0^+}x(t) & = Ax(t) +\hat{f}(x(t)), t > 0, \\
		 x(0) &= x^0.\label{ic5}
	\end{align} 
	\end{subequations}
	Then, for any $x^0\in B(0,\varepsilon_0)$, its unique solution 
	$\hat{\varphi}(\cdot,x^0)=(\hat{\varphi}_1(\cdot,x^0),\hat{\varphi}_2(\cdot,x^0))^{\rm T}$ on $[0,\infty)$
	staifies the relationships
	\begin{subequations}
	\label{eq:x1}
	\begin{align}\label{x1}
		\hat{\varphi}_1(\cdot,x^0)
		&=  \left (\mathcal R^0(t)-a_{22}\mathcal R^{\alpha_2}(t)  \right )x_1^0 + a_{12}\mathcal R^{\alpha_1}(t)x^0_2 \nonumber \\
		&\phantom{=} \quad+ \left( ( \mathcal S^{\alpha_1} - a_{22}\mathcal S^{l(\alpha)} )
			\ast \hat{f}_1(\hat{\varphi}(\cdot,x^0))  \right )(t) 
			+ a_{12}\left (\mathcal S^{l(\alpha)} \ast \hat{f}_2 (\hat{\varphi}(\cdot,x^0)) \right )(t), \\
		\label{x2}
		\hat{\varphi}_2(t,x^0) 
		&=  a_{21}\mathcal R^{\alpha_2}(t)x^0_1 + 
			\left (\mathcal R^0(t)-a_{11}\mathcal R^{\alpha_1}(t)  \right )x_2^0 \nonumber \\
		&+a_{21}\left (\mathcal S^{l(\alpha)} \ast \hat{f}_1(\hat{\varphi}(\cdot,x^0))  \right )(t) 
			+ \left (( \mathcal S^{\alpha_2}- a_{11}\mathcal S^{l(\alpha)}) \ast \hat{f}_2(\hat{\varphi}(\cdot,x^0))  \right )(t). 
	\end{align}
	\end{subequations}
	To show the Mittag-Leffler stability of the trivial solution to the original system, 
	we will prove that for any small initial value vector, the unique solution of the system 
	\eqref{eq:5} is contained in the space $C_{\infty}([0,\infty);\mathbb R^2)$ which is equipped the norm  
	\[
		\|\xi\|_w:=\max\{ \sup_{t\in [0,1]}\|\xi(t)\|, \sup_{t\geq 1}t^\nu \|\xi(t)\| \}.
	\]
	It is easy to see that $C_w([0,\infty);\mathbb R^2):=\{\xi\in C_{\infty}([0,\infty);\mathbb R^2):\|\xi\|_w<\infty\}$ 
	is a Banach space with the norm $\|\cdot\|_w$. For $\varepsilon > 0$, let $B_{C_w}(0,\varepsilon) :=
	\{\xi\in C_{\infty}([0,\infty);\mathbb R^2):\|\xi\|_w\leq \varepsilon\}$.
	
	Based on the representation in \eqref{eq:x1}, we establish a Lyapunov--Perron type operator $\mathcal{T}_{x^0}$ 
	on the space $C_w([0,\infty);\mathbb R^2)$ as follows. For any $\xi\in C_w([0,\infty);\mathbb R^2)$, let
	\begin{align*}
		(\mathcal{T}_{x_0}\xi)_1(t)
		&:= \left (\mathcal R^0(t)-a_{22}\mathcal R^{\alpha_2}(t)  \right )x_1^0 + a_{12}\mathcal R^{\alpha_1}(t)x^0_2 \nonumber \\
		& \phantom{= } \quad + \left( ( \mathcal S^{\alpha_1} - a_{22}\mathcal S^{l(\alpha)} )\ast \hat{f}_1(\xi(\cdot))  \right )(t) 
			+ a_{12}\left (\mathcal S^{l(\alpha)}\ast\hat{f}_2 (\xi(\cdot)) \right )(t), \\
		(\mathcal{T}_{x_0}\xi)_2(t)
		&:= a_{21}\mathcal R^{\alpha_2}(t)x^0_1 + \left (\mathcal R^0(t)-a_{11}\mathcal R^{\alpha_1}(t)  \right )x_2^0 \nonumber \\
		&\phantom{= } \quad + a_{21}\left (\mathcal S^{l(\alpha)}\ast\hat{f}_1(\xi(\cdot))  \right )(t) 
			+\left (( \mathcal S^{\alpha_2}- a_{11}\mathcal S^{l(\alpha)})\ast\hat{f}_2(\xi(\cdot))\right )(t). 
	\end{align*}
	On the interval $[0,1]$, we have 
	\begin{align}
		|(\mathcal{T}_{x^0}\xi)_1(t)|
		&\leq (|\mathcal R^0(t)|+|a_{22}| \cdot |\mathcal R^{\alpha_2}(t)| )|x_1^0| 
			+ |a_{12}| \cdot |\mathcal R^{\alpha_1}(t)| \cdot |x^0_2|\nonumber \\
		&\phantom{= } \quad + l_{\hat{f}}(\|\xi\|_\infty)\|\xi\|_w \int_0^t \left(|\mathcal S^{\alpha_1}(s)|
			+ (|a_{22}|+|a_{12}|)|\mathcal S^{l(\alpha)}(s)|\right)ds\notag\\
		&\le \sup_{t\in[0,1]} \left(\left(|\mathcal R^0(t)|+|a_{22}| \cdot |\mathcal R^{\alpha_2}(t)|  \right )|x_1^0| 
			+ |a_{12}| \cdot |\mathcal R^{\alpha_1}(t)| \cdot |x^0_2| \right)\notag\\
		&\phantom{= } \quad + l_{\hat f}(\|\xi\|_\infty)\|\xi\|_w C (1+|a_{22}|+|a_{12}|)\int_0^1 \frac{1}{s^{1-\nu}}ds\notag\\
		&\leq \sup_{t\in[0,1]}\left(\left(|\mathcal R^0(t)|+|a_{22}| \cdot |\mathcal R^{\alpha_2}(t)|  \right )|x_1^0| 
			+ |a_{12}| \cdot |\mathcal R^{\alpha_1}(t)| \cdot |x^0_2|\right)\notag\\
		&\phantom{= } \quad + l_{\hat f}(\|\xi\|_\infty)\|\xi\|_w \frac{C(1+|a_{22}|+|a_{12}|)}{\nu},\label{op_1} \\
		|(\mathcal{T}_{x^0}\xi)_2(t)|
		&\leq |a_{21}| \cdot |\mathcal R^{\alpha_2}(t)| \cdot |x^0_1|
			+(|\mathcal{R}^0(t)|+|a_{11}| \cdot |\mathcal R^{\alpha_1}(t)|) |x_2^0|\nonumber \\
		&\phantom{= } \quad + l_{\hat{f}}(\|\xi\|_\infty)\|\xi\|_w \int_0^t \left(|\mathcal S^{\alpha_2}(s)|
			+ (|a_{12}|+|a_{11}|)|\mathcal S^{l(\alpha)}(s)|\right)ds\notag\\
		&\le\sup_{t\in[0,1]} \left(|a_{21}| \cdot |\mathcal R^{\alpha_2}(t)| \cdot |x^0_1|+(|\mathcal{R}^0(t)|
			+|a_{11}| \cdot |\mathcal R^{\alpha_1}(t)|)|x_2^0| \right)\notag\\
		&\phantom{= } \quad + l_{\hat f}(||\xi||_\infty)||\xi||_wC(1+|a_{21}|+|a_{11}|)\int_0^1 \frac{1}{s^{1-\nu}}ds\notag\\
		&\leq \sup_{t\in[0,1]} \left (|a_{21}| \cdot |\mathcal R^{\alpha_2}(t)| \cdot |x^0_1|
			+(|\mathcal{R}^0(t)|+|a_{11}| \cdot |\mathcal R^{\alpha_1}(t)|)|x_2^0| \right)\notag\\
		&\phantom{= } \quad + l_{\hat f}(\|\xi\|_\infty)\|\xi\|_w\frac{C(1+|a_{21}|+|a_{11}|)}{\nu}.\label{op_2}
	\end{align}
	For $t\in[1,\infty)$, then
	\begin{align}
	\notag 	t^\nu |( \mathcal{T}_{x_0}\xi)_1(t)|&\le \sup_{t\geq 1}\Big(\left(t^\nu|\mathcal R^0(t)|+|a_{22}|t^\nu|\mathcal R^{\alpha_2}(t)|  \right )|x_1^0| + |a_{12}|t^\nu|\mathcal R^{\alpha_1}(t)||x^0_2|\Big)\notag\\
	\notag &\hspace{-1cm}+l_{\hat f}(||\xi||_\infty)t^\nu\int_0^t \left(|\mathcal S^{\alpha_1}(t-s)|+ (|a_{22}|+|a_{12}|)|\mathcal S^{l(\alpha)}(t-s)|\right)s^{-\nu}s^{\nu}|\xi(s)|ds\notag\\
		\notag &\leq \sup_{t\geq 1}\Big(\left(t^\nu|\mathcal R^0(t)|+|a_{22}|t^\nu|\mathcal R^{\alpha_2}(t)|  \right )|x_1^0| + |a_{12}|t^\nu|\mathcal R^{\alpha_1}(t)||x^0_2|\Big)\notag\\
		&\hspace{-2cm} +l_{\hat f}(||\xi||_\infty)\|\xi\|_w \sup_{t\geq 1}t^\nu\int_0^t  \left(|\mathcal S^{\alpha_1}(t-s)|+ (|a_{22}|+|a_{12}|)|\mathcal S^{l(\alpha)}(t-s)|\right)s^{-\nu}ds,\label{op_12}
	\end{align}
	\begin{align}
	\notag 	t^\nu |( \mathcal{T}_{x_0}\xi)_2(t)|&\le \sup_{t\geq 1}\Big(|a_{21}|t^\nu|\mathcal R^{\alpha_2}(t)||x^0_1|+(t^\nu|\mathcal{R}^0(t)|||+|a_{11}|t^\nu|\mathcal R^{\alpha_1}(t)|)|x_2^0|\Big)\notag\\
	\notag &\hspace{-1cm}+l_{\hat f}(||\xi||_\infty)t^\nu\int_0^t \left(|\mathcal S^{\alpha_2}(t-s)|+ (|a_{21}|+|a_{11}|)|\mathcal S^{l(\alpha)}(t-s)|\right)s^{-\nu}s^\nu \|\xi(s)\|ds \notag\\
		\notag &\leq \sup_{t\geq 1}\Big(|a_{21}|t^\nu|\mathcal R^{\alpha_2}(t)||x^0_1|+(t^\nu|\mathcal{R}^0(t)|||+|a_{11}|t^\nu|\mathcal R^{\alpha_1}(t)|)|x_2^0|\Big)\notag\\
		&\hspace{-2cm} +l_{\hat f}(||\xi||_\infty)\|\xi\|_w \sup_{t\geq 1}t^\nu\int_0^t \left(|\mathcal S^{\alpha_2}(t-s)|+ (|a_{21}|+|a_{11}|)|\mathcal S^{l(\alpha)}(t-s)|\right)s^{-\nu} ds.\label{op_22}
	\end{align}
	From \eqref{op_12} and \eqref{op_22}, we obtain the estimates
	\begin{align}
		\label{op_3}	
		\|(\mathcal{T}_{x^0}\xi)_1\|_{w,1}
		&\leq (\|\mathcal{R}^0\|_{w,1}+|a_{22}| \cdot \|\mathcal{R}^{\alpha_2}\|_{w,1}
			+|a_{12}| \cdot \|\mathcal{R}^{\alpha_1}\|_{w,1})\|x^0\|\\
		&\phantom{=} \quad +l_{\hat f}(||\xi||_\infty)\|\xi\|_w \left(\frac{C}{\nu}(1+|a_{12}|+|a_{22}|)+M_{\alpha_1}
			+(|a_{22}|+|a_{12}|)M_{l(\alpha)}\right) 
		\nonumber
	\end{align}
	and
	\begin{align}
		\label{op_4}
		\|(\mathcal{T}_{x^0}\xi)_2\|_{w,1}
		&\leq (\|\mathcal{R}^0\|_{w,1}+|a_{21}| \cdot \|\mathcal{R}^{\alpha_2}\|_{w,1}
			+|a_{11}| \cdot \|\mathcal{R}^{\alpha_1}\|_{w,1})\|x^0\|\\
		&\phantom{=} \quad+l_{\hat f}(||\xi||_\infty)\|\xi\|_w \left(\frac{C}{\nu}(1+|a_{21}|+|a_{11}|)
			+M_{\alpha_2}+(|a_{21}|+|a_{11}|)M_{l(\alpha)}\right)
			\nonumber
	\end{align}
	where $\|\xi\|_{w,1}:=\max\{\sup_{t\in [0,1]}|\xi(t)|,\sup_{t\geq 1}t^\nu |\xi(t)|\}$ 
	for any $\xi\in C_{\infty}([0,\infty);\mathbb R)$, 
	$M_{\beta}=\sup_{t\geq 1} t^\nu \int_0^t |\mathcal S^{\beta}(t-s)| s^{-\nu}ds$,
	for $\beta \in \{ \alpha_1, \alpha_2,  l(\alpha) \}$, and	
	By \eqref{op_3} and \eqref{op_4}, we see that
	\begin{align}
		\notag \|(\mathcal{T}_{x^0}\xi)\|_{w}
		&\leq (2\|\mathcal{R}^0\|_{w,1}+(|a_{11}|+|a_{12}|)\|\mathcal{R}^{\alpha_1}\|_{w,1}
			+(|a_{21}|+|a_{22}| \cdot \|\mathcal{R}^{\alpha_2}\|_{w,1})\|x^0\|\\
		\notag 
		& + l_{\hat f}(||\xi||_\infty)\|\xi\|_w\left(\frac{C}{\nu}\left(2+\sum_{i,j=1}^2|a_{ij}|\right)+M_{\alpha_1}
			+M_{\alpha_2}+\sum_{i,j=1}^2|a_{ij}|M_{l(\alpha)}\right).
	\end{align}
	On the other hand, by virtue of the assumption that $\lim_{r\to 0}l(r)=0$, we can choose $\varepsilon\in (0,\varepsilon_0)$ so that
	\[
		r_0 := \left[\frac{2C}{\nu} + M_{\alpha_1}+M_{\alpha_2}
			+\sum_{i,j=1}^2 |a_{ij}| \left (M_{l(\alpha)}+\frac{C}{\nu}\right )\right]l_{\hat f}(\varepsilon)<1.
	\]
	Take $\delta=\varepsilon(1-r) / \left(2\|\mathcal{R}^0\|_{w,1}+(|a_{11}|+|a_{12}|)\|\mathcal{R}^{\alpha_1}\|_{w,1}
	+(|a_{21}|+|a_{22}|)\|\mathcal{R}^{\alpha_2}\|_{w,1} \right)$, then for any initial condition $x^0\in B(0,\delta)$, we have
	\[
		\|\mathcal{T}_{x^0}\xi\|_w\leq \varepsilon,\;\forall \xi\in B_{C_w}(0,\varepsilon),
	\]
	that is, $\mathcal{T}_{x^0}(B_{C_w}(0,\varepsilon))\subset B_{C_w}(0,\varepsilon)$. 
	Moreover, for every $\xi,\hat{\xi}\in B_{C_w}(0,\varepsilon)$, 
	\begin{align}
		\|\mathcal{T}_{x^0}\xi-\mathcal{T}_{x^0}\hat{\xi}\|_w
		& \leq \left[\frac{2C}{\nu}+M_{\alpha_1}+M_{\alpha_2}
			+ \sum_{i,j=1}^2 |a_{ij}|(M_{l(\alpha)} + \frac{C}{\nu})\right]l_{\hat f}(\varepsilon)\|\xi-\hat{\xi}\|_w\\
		&=r_0 \|\xi-\hat{\xi}\|_w.
	\end{align}
	Thus, the operator $\mathcal{T}_{x^0}$ is contractive on $B_{C_w}(0,\varepsilon)$, and by Banach's fixed point theorem, 
	$\mathcal{T}_{x^0}$ has a unique fixed point $\xi^*$ in this set. Furthermore, this function is the unique solution to the system \eqref{eq:5}
	in $B_{C_w}(0,\varepsilon)$. Notice that if $\xi^*\in B_{C_w}(0,\varepsilon)$ then $f(\xi^*(t))=\hat{f}(\xi^*(t))$ for every $t\in[0,\infty)$,
	and thus $\xi^*$ is also a solution to the system \eqref{eq:4}. This completes the proof.
\end{proof}

\section{Numerical examples}
\label{sec:examples}

To complete this paper, we now give some numerical examples to illustrate the main theoretical results.
In all the examples below, we use the functions $f_1$ and $f_2$ with
\[ 
	f_i(t) = \begin{cases}
			 1& \text{ if } 0 \leq t < 1, \\
			 \frac{1}{t^{2i}}& \text{ if } t \geq 1
		\end{cases} 
		\quad (i = 1, 2).
\]
For all cases, we have calculated numerical solutions to verify the theoretical findings. 
These solutions have been computed with Garrappa's MATLAB implementation of 
the implicit trapezoidal method described in detail in \cite{Garrappa}. This algorithm is known 
to have very favourable stability properties which makes it highly suitable for handling 
equations like ours over large intervals (which is required in this case to demonstrate 
the asymptotic behaviour). The step size has always been chosen as $h = 1/200$.

\begin{example} 
	\label{ex:1}
	Consider the inhomogeneous two-component incommensurate fractional-order linear system
	\begin{align}\label{ex1}
		\begin{cases}
			^C D^{1/3}_{0^+}x_1(t) = 0.25x_2(t) + f_1(t), \\
			^C D^{1/2}_{0^+}x_2(t) = -2x_1(t) + x_2(t)+ f_2(t), 
		\end{cases}
		\quad t>0.
	\end{align}
	In this example, the characteristic function is $Q(s) = s^{5/6} + s^{1/2} + 0.5.$  
	According to Lemma \ref{bd3.3}, all zeros of $Q$ lie in the open left-half of the complex plane. 
	Furthermore, the function $f$ satisfies the assumption stated in Theorem \ref{dl3}. 
	Hence, every solution to \eqref{ex1} tends to the origin as $t\to\infty$ with the rate $O(t^{-1/3})$. 
	This property is illustrated in Figure \ref{fig:1}. The left graph shows that the components $x_1(t)$
	and $x_2(t)$ decay to zero; the right graph visualizes the fact that $t^{1/3} x_j(t)$ tends to
	a nonzero constant for $t \to \infty$ and $j = 1, 2$, thus demonstrating that the decay
	behaviour of $x_j(t)$ is indeed $O(t^{-1/3})$.
\end{example}

\begin{figure}[h]
	\includegraphics[width=0.48\textwidth]{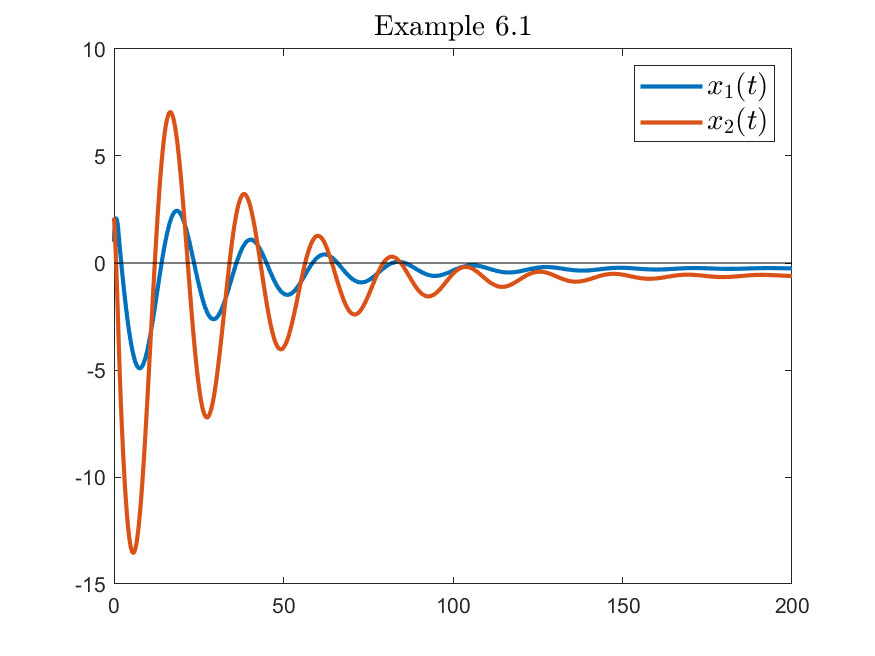}
	\hfill
	\includegraphics[width=0.48\textwidth]{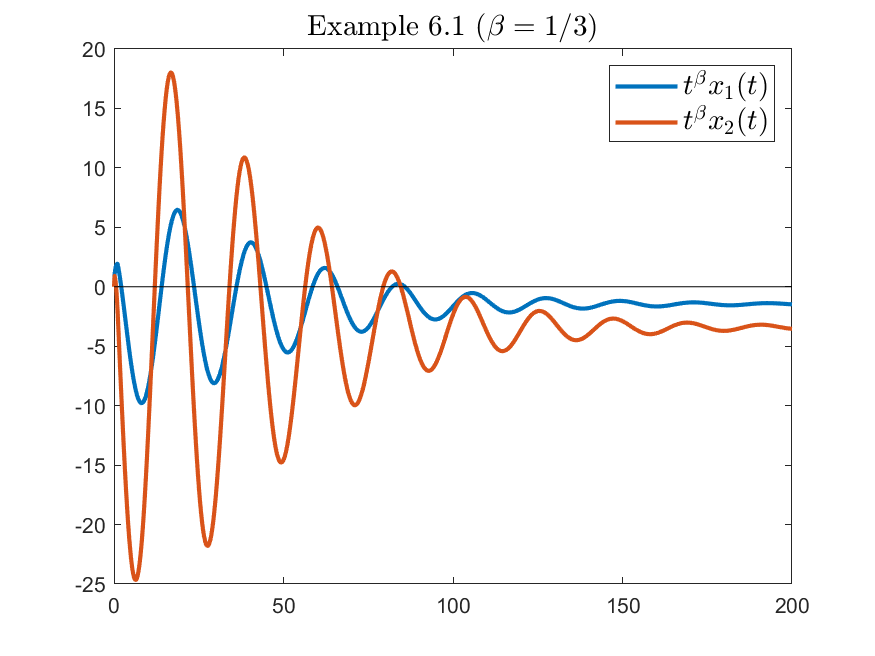}

	\caption{\label{fig:1}Solution to the differential equation \eqref{ex1} from Example \ref{ex:1} 
		with initial conditions $x_1(0) = 1$ and $x_2(0) = 2$.
		The left graph shows the components $x_j(t)$ of the solutions themselves, the right graph 
		shows the functions $t^{1/3} x_j(t)$.}
\end{figure}

\begin{example} 
	\label{ex:2}
	Consider the two-component incommensurate fractional-order nonlinear system
	\begin{align}\label{ex2}
		\begin{cases}
			^C D^{1/3}_{0^+}x_1(t) = 0.25x_2(t) + x_1^2(t)x_2^2(t), \\
			^C D^{1/2}_{0^+}x_2(t) = -2x_1(t) + x_2(t)+ x_1^2(t) + x_2^2(t),
		\end{cases}
		\quad t>0. 
	\end{align}
	It is not difficult to check that all conditions of Lemma \ref{bd3.3} and Theorem \ref{dl4.4} are verified. 
	Thus, the trivial solution to \eqref{ex2} is Mittag-Leffler stable; more precisely, by Theorem \ref{dl4.4},
	we have to expect an $O(t^{-1/3})$ decay behavior for nontrivial solutions with initial values
	sufficiently close to those of the trivial solution.
	
	Defintion \ref{def:mlstable} states that the boundedness of the solutions cannot be expected for all choices
	of the initial value any more (as had been the case in Example \ref{ex:1}) but only for initial values
	sufficiently close to $(0,0)$. Indeed we can see this behaviour in Figure \ref{fig:2} for the
	initial value $(0.1, -0.2)$, whereas Figure \ref{fig:2x} shows that this behaviour is not present
	for initial values farther away from $(0,0)$ such as, e.g., the initial value $(1,-1)$.
	In the latter case, the solutions still seem to be bounded, but the decay behaviour appears to
	be absent. If one moves the initial values even farther away from the equilibrium point,
	then one cannot even expect this boundedness any more.
\end{example}

\begin{figure}[h]
	\includegraphics[width=0.48\textwidth]{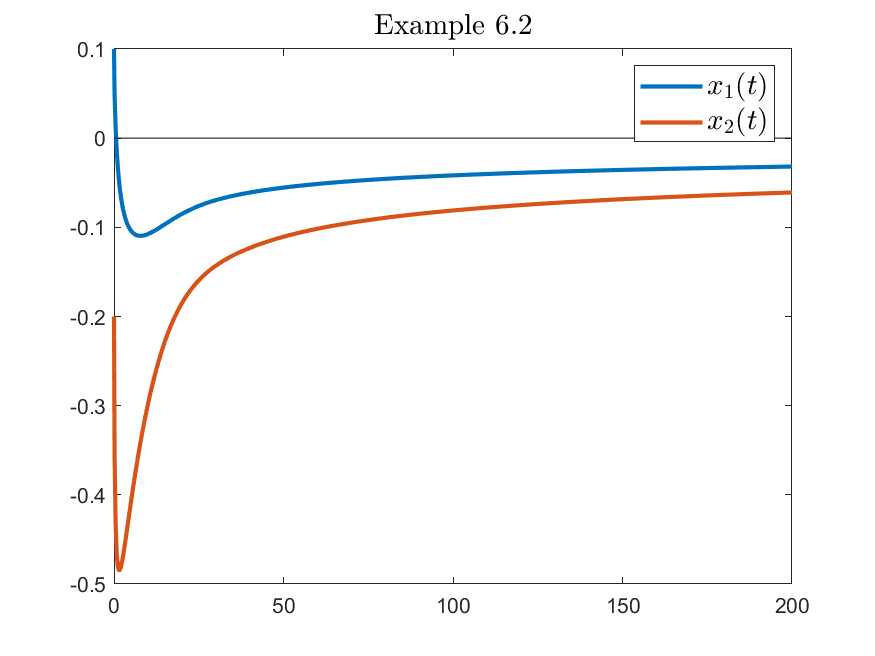}
	\hfill
	\includegraphics[width=0.48\textwidth]{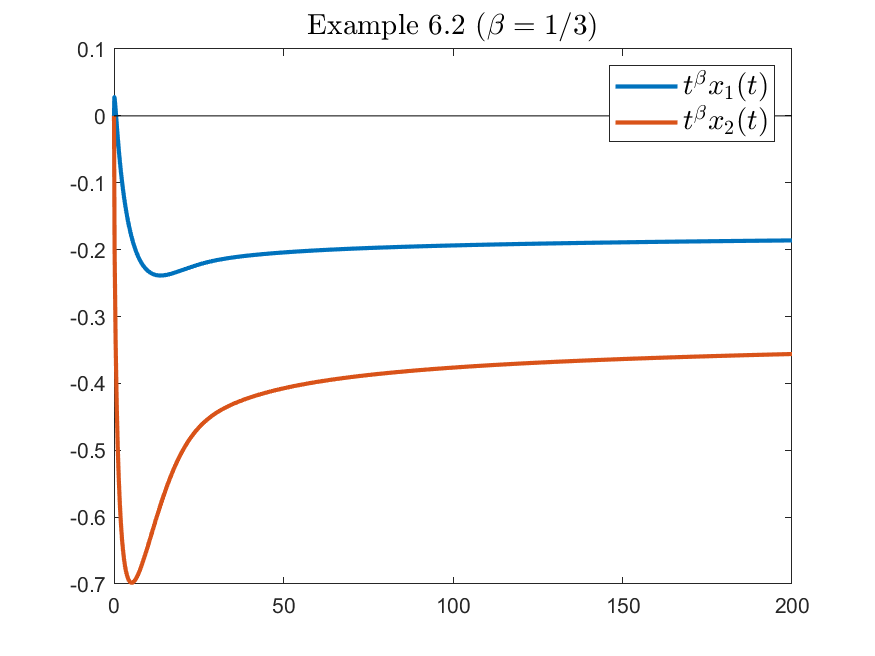}

	\caption{\label{fig:2}Solution to the differential equation \eqref{ex2} from Example \ref{ex:2} 
		with initial conditions $x_1(0) = 0.1$ and $x_2(0) = -0.2$.
		The left graph shows the components $x_j(t)$ of the solutions themselves, the right graph 
		shows the functions $t^{1/3} x_j(t)$.}
\end{figure}

\begin{figure}[h]
	\includegraphics[width=0.48\textwidth]{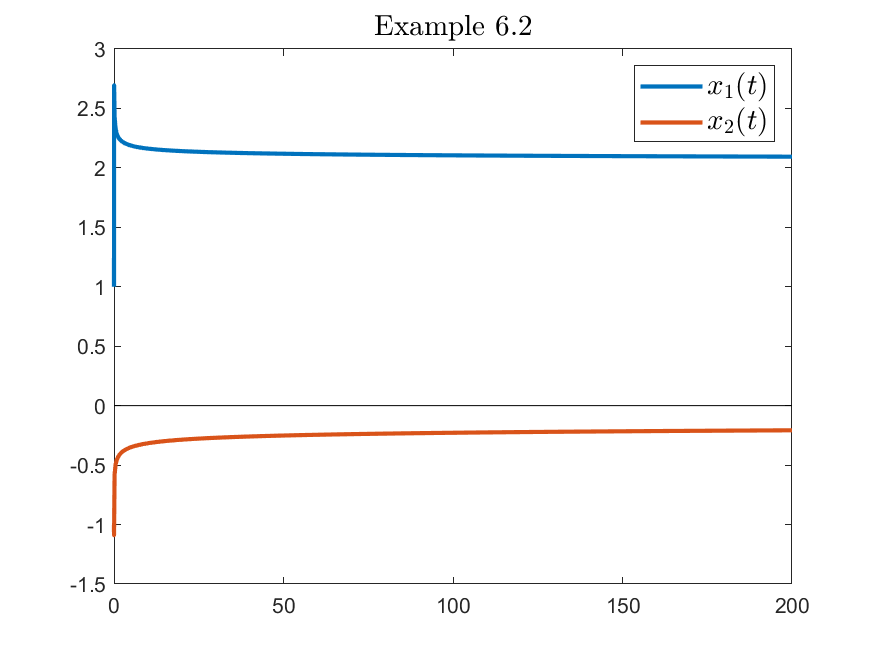}
	\hfill
	\includegraphics[width=0.48\textwidth]{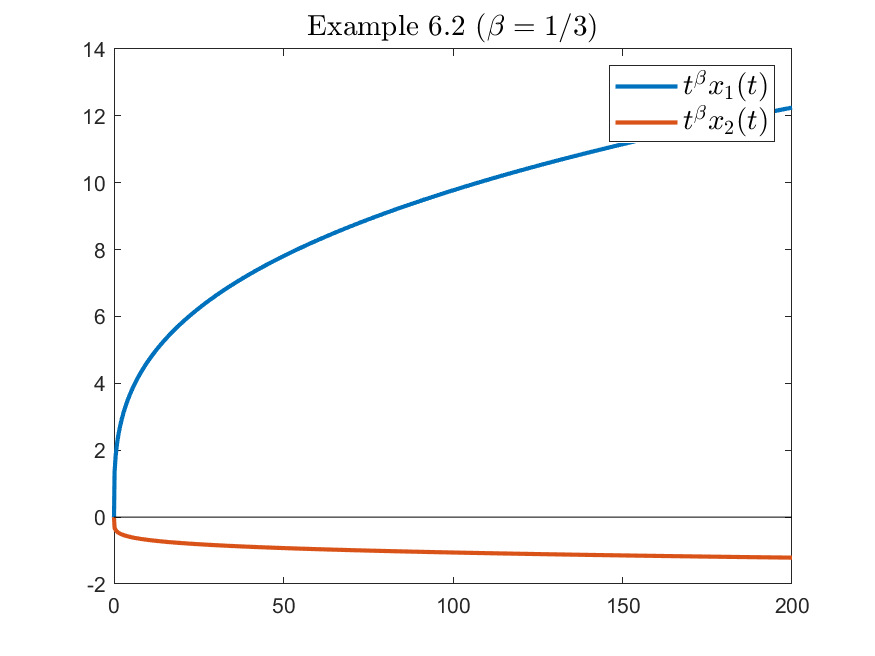}

	\caption{\label{fig:2x}Solution to the differential equation \eqref{ex2} from Example \ref{ex:2} 
		with initial conditions $x_1(0) = 1$ and $x_2(0) = -1$.
		The left graph shows the components $x_j(t)$ of the solutions themselves, the right graph 
		shows the functions $t^{1/3} x_j(t)$.}
\end{figure}

\begin{example} 
	\label{ex:3}
	Consider the fractional linear system
	\begin{align}\label{vd3}
		\begin{cases}
			^C D^{0.6}_{0^+}x_1(t) =x_1(t) + 2x_2(t) + f_1(t), \\
			^C D^{0.8}_{0^+}x_2(t) = -x_1(t) + f_2(t),
		\end{cases}
		\quad t>0. 
	\end{align}
	The characteristic function of the system is $Q(s) = s^{1.4} + s^{0.6} + 2.$  By Lemma \ref{bd3.4}, 
	all zeros of $Q$ lie in the open left-half of the complex plane and the assumptions of Theorem \ref{dl3} 
	are satisfied. Hence, every solution to this system converges to the origin as $t\to\infty$ with an $O(t^{-0.6})$
	convergence rate. 
	As in Example \ref{ex:1}, we can also reproduce this behaviour numerically. The corresponding graphs 
	are plotted in Figure \ref{fig:3}.
\end{example}

\begin{figure}[h]
	\includegraphics[width=0.48\textwidth]{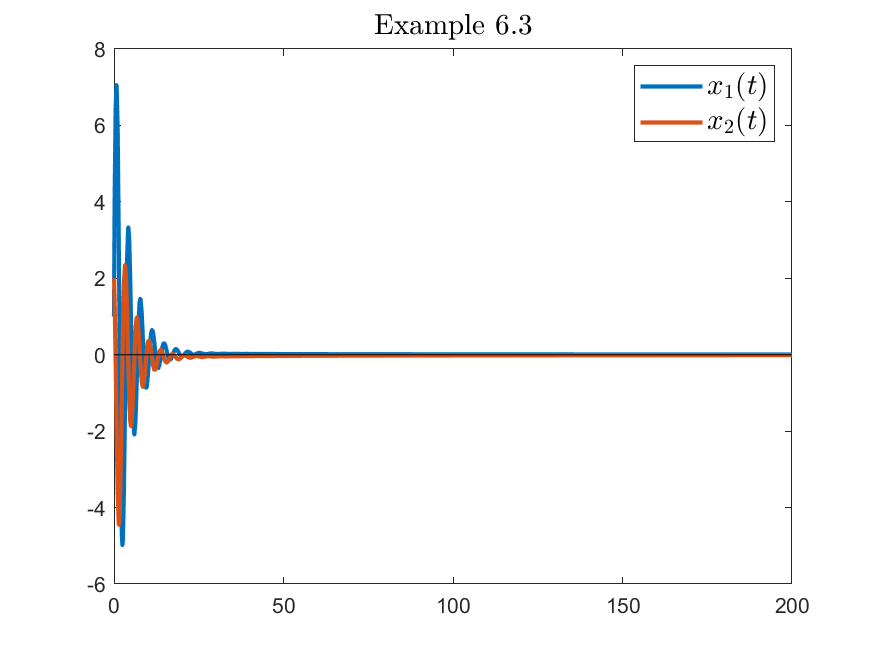}
	\hfill
	\includegraphics[width=0.48\textwidth]{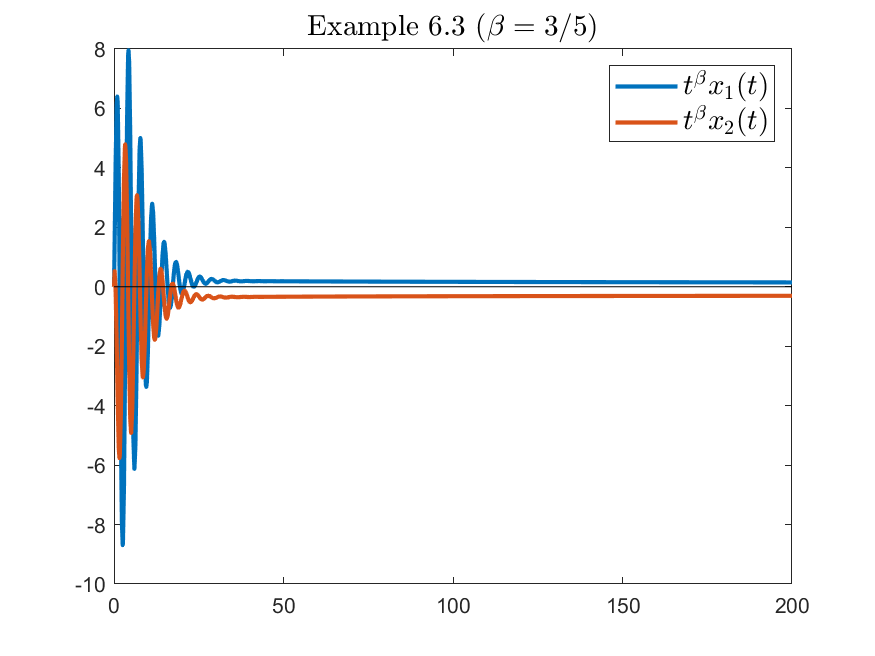}

	\caption{\label{fig:3}Solution to the differential equation \eqref{vd3} from Example \ref{ex:3} 
		with initial conditions $x_1(0) = 1$ and $x_2(0) = 2$.
		The left graph shows the components $x_j(t)$ of the solutions themselves, the right graph 
		shows the functions $t^{0.6} x_j(t)$.}
\end{figure}

\begin{example} 
	\label{ex:4}
	Consider the system
	\begin{align}\label{vd4}
		\begin{cases}
			^C D^{0.6}_{0^+}x_1(t) =x_1(t) + 2x_2(t) + x_1^2(t)x_2^2(t),\\
			^C D^{0.8}_{0^+}x_2(t) = -x_1(t) + x_1^2(t) + x_2^2(t),
		\end{cases}
		\quad t > 0.
	\end{align}
	Based on Lemma \ref{bd3.4} and Theorem \ref{dl4.4}, we see that the trivial solution of \eqref{vd4} is Mittag-Leffler stable. 
	As in Example \ref{ex:2}, this is exhibited---together with the decay behaviour predicted by Theorem \ref{dl4.4}---
	in Figure \ref{fig:4}.
\end{example}

\begin{figure}[h]
	\includegraphics[width=0.48\textwidth]{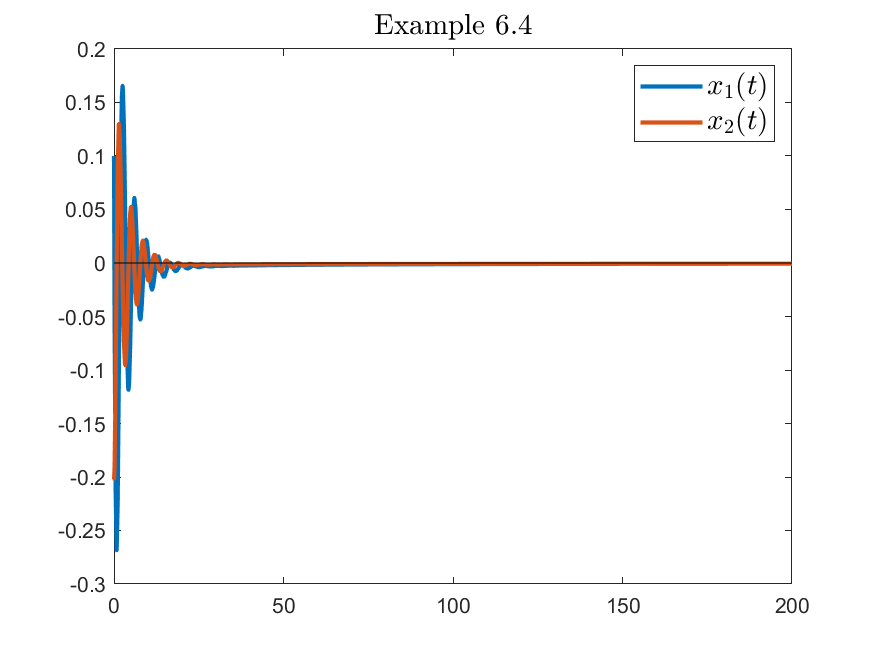}
	\hfill
	\includegraphics[width=0.48\textwidth]{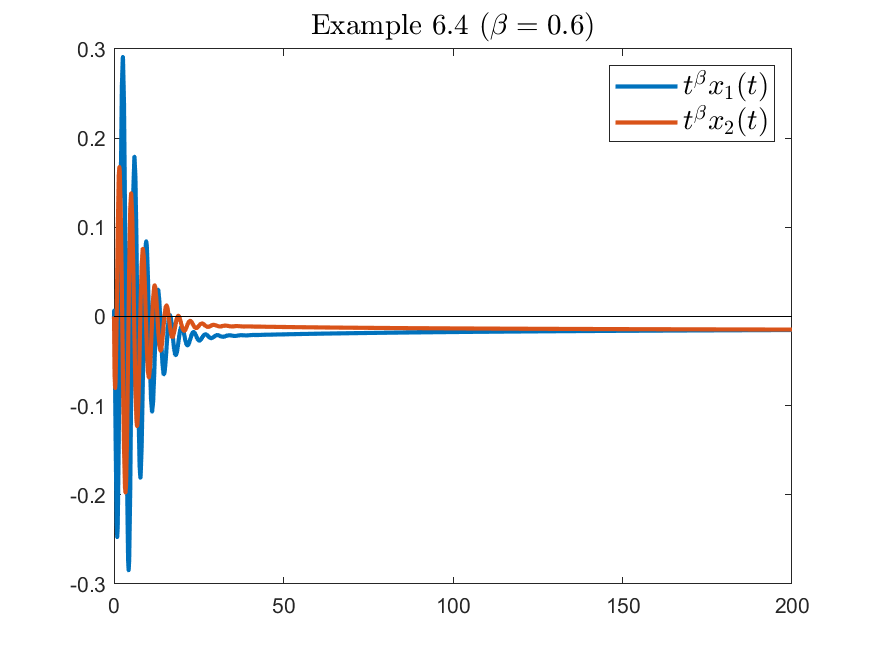}

	\caption{\label{fig:4}Solution to the differential equation \eqref{vd4} from Example \ref{ex:4} 
		with initial conditions $x_1(0) = 0.1$ and $x_2(0) = -0.2$.
		The left graph shows the components $x_j(t)$ of the solutions themselves, the right graph 
		shows the functions $t^{0.6} x_j(t)$.}
\end{figure}

\begin{example}
	\label{ex:5} 
	Consider the inhomogeneous two-component incommensurate fractional-order linear system
	\begin{align}\label{vd5}
		\begin{cases}
			^C D^{0.3}_{0^+}x_1(t) =x_1(t) - x_2(t) + f_1(t) , \\
			^C D^{0.4}_{0^+}x_2(t) = 2x_1(t) +x_2(t)+ f_2(t),
		\end{cases}
		\quad t > 0.
	\end{align}
	The system \eqref{vd5} has the characteristic function $Q(s) = s^{0.7} + s^{0.4}+s^{0.3} + 3.$  
	From Lemma \ref{bd3.5} (i) and Theorem \ref{dl3}, it follows that every solution of this system tends to 
	the origin as $t\to\infty$ as $O(t^{-0.3})$. 
	Once again, our numerical results, shown in Figure \ref{fig:5}, support this statement.
\end{example}

\begin{figure}[h]
	\includegraphics[width=0.48\textwidth]{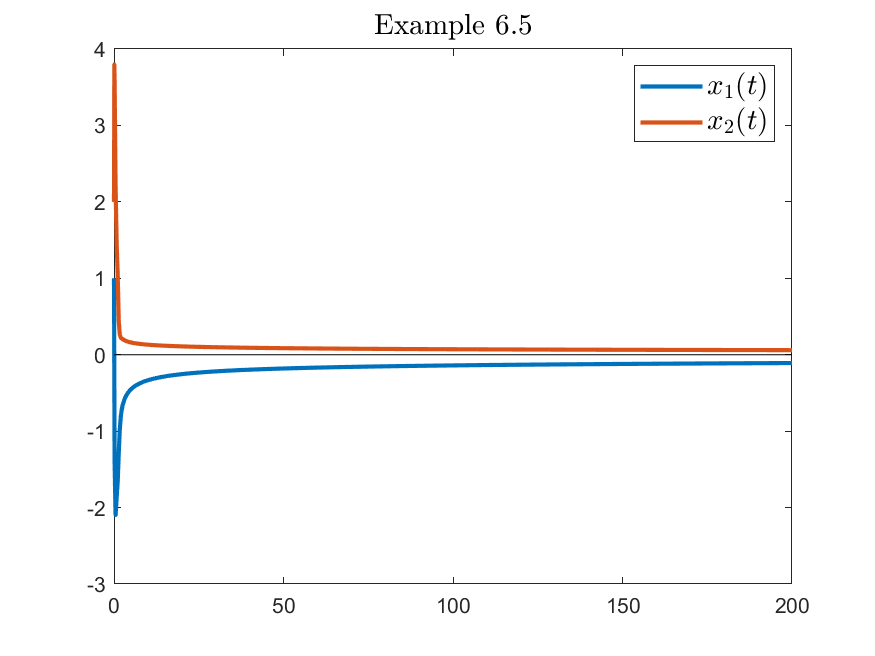}
	\hfill
	\includegraphics[width=0.48\textwidth]{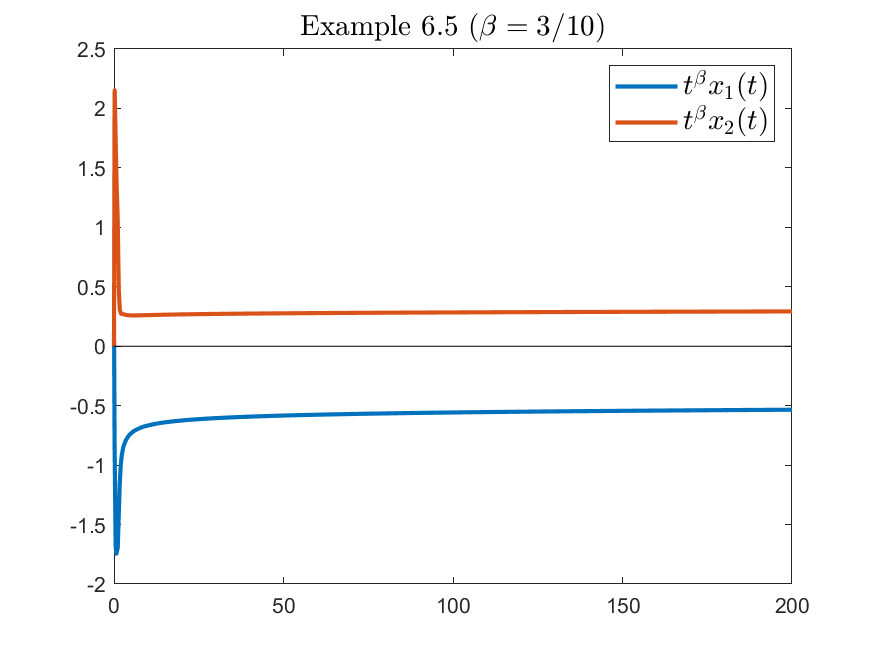}

	\caption{\label{fig:5}Solution to the differential equation \eqref{vd5} from Example \ref{ex:5} 
		with initial conditions $x_1(0) = 1$ and $x_2(0) = 2$.
		The left graph shows the components $x_j(t)$ of the solutions themselves, the right graph 
		shows the functions $t^{0.3} x_j(t)$.}
\end{figure}

\begin{example}
	\label{ex:6} 
	Consider the two-component incommensurate fractional-order nonlinear system
	\begin{align}\label{vd6}
		\begin{cases}
			^C D^{0.3}_{0^+}x_1(t) =0.1x_1(t) - 0.4x_2(t) + x_1^2(t)x_2^2(t),\\
			^C D^{0.4}_{0^+}x_2(t) = 0.7x_1(t) +0.2x_2(t) +  x_1^2(t) + x_2^2(t),
		\end{cases}
		\quad t > 0.
	\end{align}
	Its characteristic function is $Q(s) = s^{0.7} + 0.2s^{0.4}+0.1s^{0.3} + 0.3.$  
	It follows from Lemma \ref{bd3.5}(ii) and Theorem \ref{dl4.4} that the trivial solution is Mittag-Leffler stable. 
	Once again, we can visualize this observarion on the basis of numerical results,
	cf.\ Figure \ref{fig:6}.
\end{example}

\begin{figure}[h]
	\includegraphics[width=0.48\textwidth]{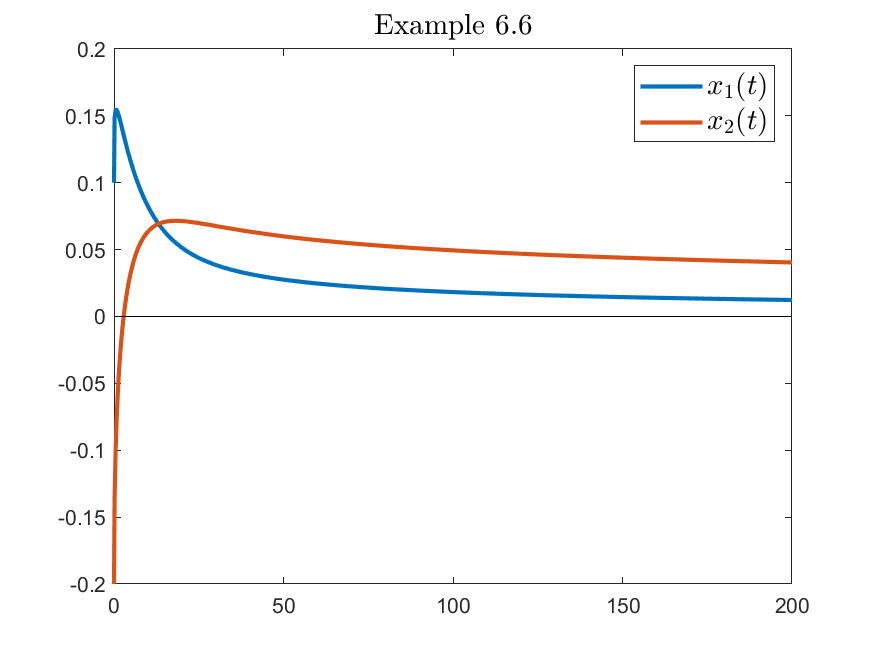}
	\hfill
	\includegraphics[width=0.48\textwidth]{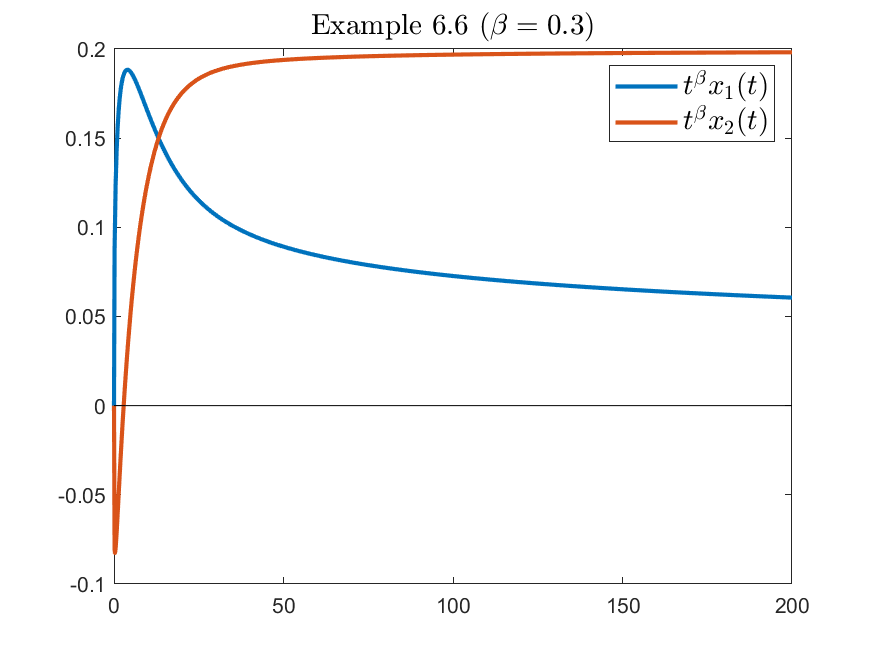}

	\caption{\label{fig:6}Solution to the differential equation \eqref{vd6} from Example \ref{ex:6} 
		with initial conditions $x_1(0) = 0.1$ and $x_2(0) = -0.2$.
		The left graph shows the components $x_j(t)$ of the solutions themselves, the right graph 
		shows the functions $t^{0.3} x_j(t)$.}
\end{figure}

\begin{example}
	\label{ex:7}
	Consider the two-component incommensurate fractional-order linear system
	\begin{align}\label{vd7}
		\begin{cases}
			^C D^{0.4}_{0^+}x_1(t) =-x_1(t) + 2x_2(t) + f_1(t),\\
			^C D^{0.5}_{0^+}x_2(t) = -5x_1(t) +4x_2(t)+ f_2(t), 
		\end{cases}
		\quad t>0.
	\end{align}
	The system \eqref{vd7} has the characteristic function $Q(s) = s^{0.9} + 4s^{0.5}-s^{0.3} + 6.$  
	According to Lemma \ref{bd3.6}(i) and Theorem \ref{dl3}, its solution converges to the origin with a rate $O(t^{-0.4})$. 
	As above, the numerical data shown in Figure \ref{fig:7} confirms this theoretical observation.
\end{example}

\begin{figure}[h]
	\includegraphics[width=0.48\textwidth]{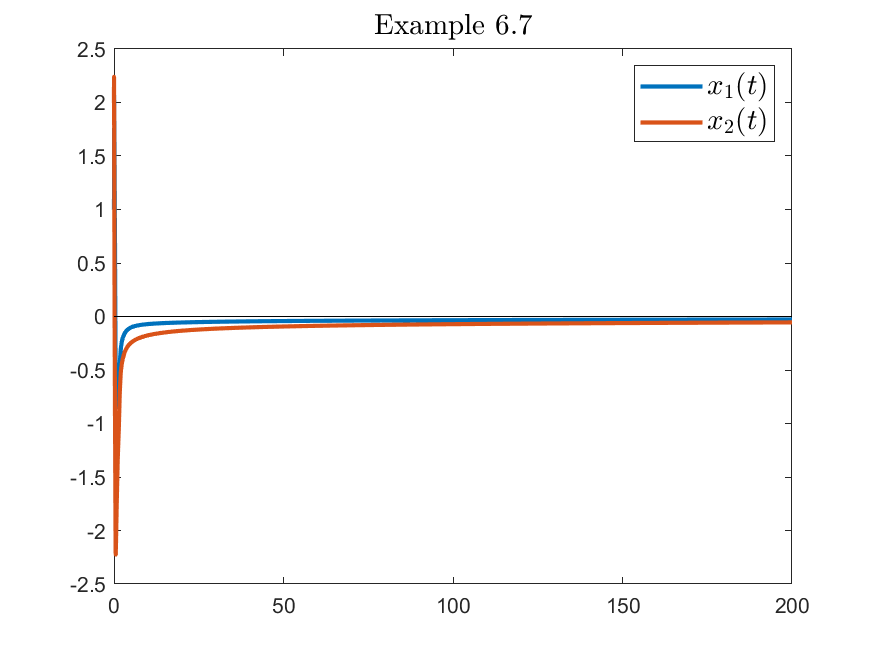}
	\hfill
	\includegraphics[width=0.48\textwidth]{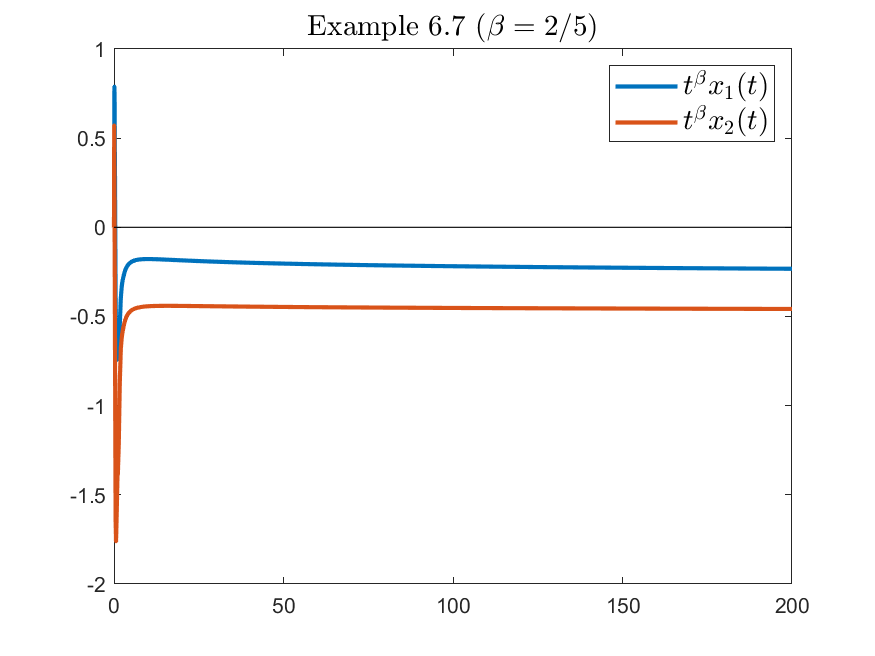}

	\caption{\label{fig:7}Solution to the differential equation \eqref{vd7} from Example \ref{ex:7} 
		with initial conditions $x_1(0) = 1$ and $x_2(0) = 2$.
		The left graph shows the components $x_j(t)$ of the solutions themselves, the right graph 
		shows the functions $t^{0.4} x_j(t)$.}
\end{figure}

\begin{example}
	\label{ex:8} 
	Consider two-component incommensurate fractional-order nonlinear system
	\begin{align}\label{vd8}
		\begin{cases}
			^C D^{0.4}_{0^+}x_1(t) =-x_1(t) - 2x_2(t) + x_1^2(t)x_2^2(t),  \\
			^C D^{0.5}_{0^+}x_2(t) = 2x_1(t) +2x_2(t) + x_1^2(t) + x_2^2(t),
		\end{cases}
		\quad t > 0.
	\end{align}
	Its characteristic function $Q(s) = s^{0.9} + 2s^{0.5}-s^{0.3} + 2$. 
	According to Lemma \ref{bd3.6}(ii) and Theorem \ref{dl4.4}, the trivial solution of \eqref{vd8} is Mittag-Leffler stable as
	illustrated graphically in Figure \ref{fig:8}.
\end{example}

\begin{figure}[h]
	\includegraphics[width=0.48\textwidth]{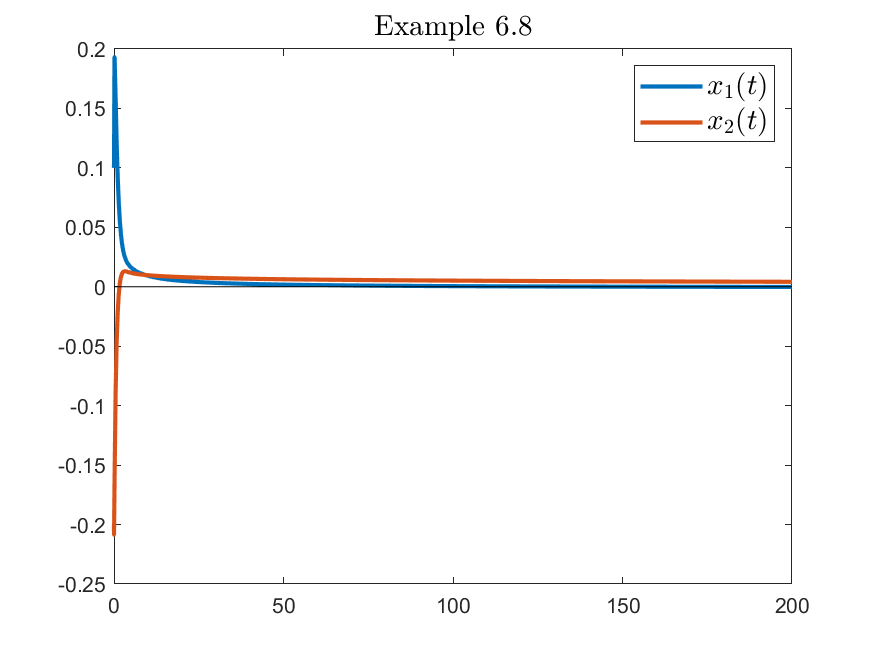}
	\hfill
	\includegraphics[width=0.48\textwidth]{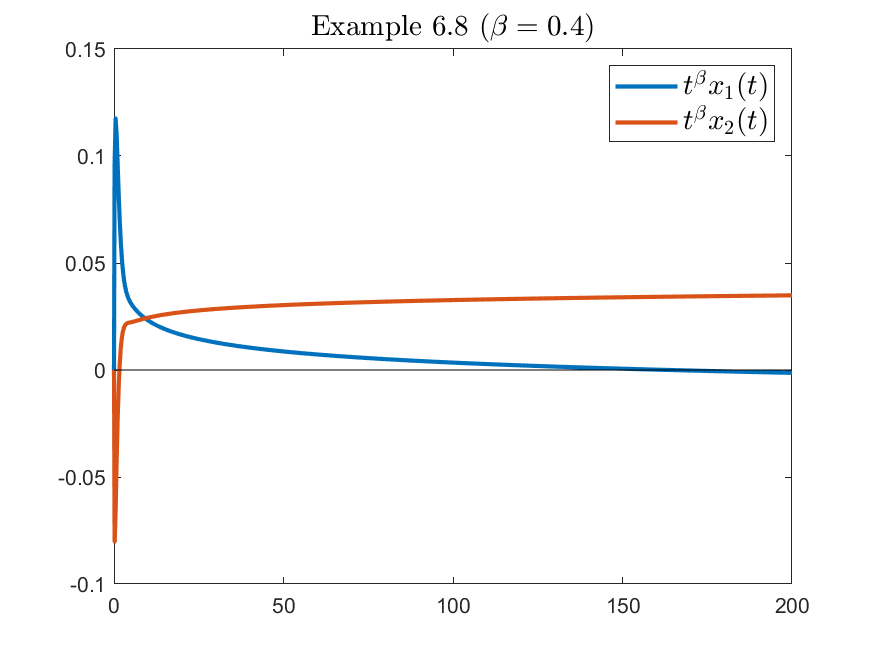}

	\caption{\label{fig:8}Solution to the differential equation \eqref{vd8} from Example \ref{ex:8} 
		with initial conditions $x_1(0) = 0.1$ and $x_2(0) = -0.2$.
		The left graph shows the components $x_j(t)$ of the solutions themselves, the right graph 
		shows the functions $t^{0.4} x_j(t)$.}
\end{figure}

\appendix

\section*{Appendix: Proof of Lemma \protect{\ref{dl4.2}}}

\begin{proof}[Proof of Lemma \ref{dl4.2}] 
	Due to the fact that there are no zeros of the characteristic function 
	$Q$ in the closed right half of the complex plane, from Lemma \ref{bd3.1}(iii), 
	we can find  $\delta > 0$ (which is small enough) such that all zeros of  $Q$ are not in the 
	domain $|\arg (s)| \leq \frac{\pi}{2} + \delta.$ Let $R > 0$ be a large enough constant such that 
	\begin{align}\label{bs2}
		|Q(s)| \geq \frac{1}{2}|s|^{\alpha_1 + \alpha_2} \quad \mbox{ whenever } |s| \geq R. 
	\end{align}
	For $\mu >0 $ and $\theta \in (0,\pi) $, we establish  an oriented contour $\gamma(\mu, \theta)$ formed
	by three segments:
	\begin{itemize}
	\item $\left \{s\in \mathbb C: |s| \geq \mu, \arg s = \theta  \right \}$,
	\item $\left \{s\in \mathbb C: |s| = \mu, |\arg s| \leq \theta  \right \}$,
	\item $\left \{s\in \mathbb C: |s| \geq \mu, \arg s = - \theta  \right \}$.
	\end{itemize}

	\noindent (i)  
	Because all zeros of $Q$ (if they exist) lie on the left of the contour $\gamma(R, \frac{\pi}{2} + \delta),$ 
	using the same argument as in \cite[Lemma 4.1]{Tuan2020FCAA}, we obtain the representation
	\begin{align}\label{7}
		\mathcal R^\lambda(t) = \frac{1}{2\pi i}\int_{\gamma(R,\frac{\pi}{2}+ \delta)} \frac{s^{l(\alpha)- \lambda-1} e^{st}}{Q(s)}ds,\; 
		t > 0,\lambda \in \left \{ 0,\alpha_1, \alpha_2 \right \}.
	\end{align}
	Choose $\varepsilon >0$ such that $Q$ has no zero in the ball $\left \{ s \in \mathbb C: |s| \leq \varepsilon \right \}.$ From \eqref{7}, we have
	\begin{align}
		\mathcal R^\lambda(t) 
		&= \frac{1}{2\pi i}\int_{\Lambda_t'} \frac{s^{l(\alpha)- \lambda-1} e^{st}}{Q(s)}ds + 
		\frac{1}{2\pi i}\int_{\gamma(\frac{\varepsilon}{t},\frac{\pi}{2}+ \delta)} 
					\frac{s^{l(\alpha)- \lambda-1} e^{st}}{Q(s)}ds, t \geq 1 \nonumber \\ 
	 	&= I_1(t)+ I_2(t),
	\end{align}
	where $\Lambda_t'$ is the clockwise oriented contour bounding the domain 
	\[
		\Omega_t :=\left \{ s \in \mathbb C: \frac{\varepsilon}{t} <|s| < R, |\arg s| < \frac{\pi}{2} + \delta \right \},
	\]
	see Figure \ref{fig:proof41a}. 
	Notice that $(s^{l(\alpha)- \lambda-1} e^{st})/Q(s)$ is analytic on 
	$\Omega_t \cup \Lambda_t'$ for all $t \geq 1$. Thus, by applying Cauchy's theorem, we obtain 
	\[ 
		I_1(t) = 0 \mbox{ for all } t \geq 1. 
	\]
	\begin{figure}[h]
		\centering
		\includegraphics[width=0.7\textwidth]{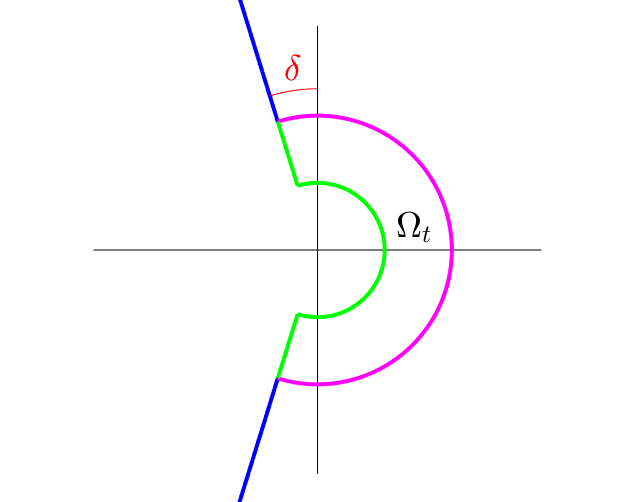}
		\caption{\label{fig:proof41a}The contours and sets used in the proof of Lemma \ref{dl4.2}: 
			The radii of the green and magenta circular arcs are $\varepsilon / t$ and $R$, respectively.
			The contour $\gamma(R, \pi/2 + \delta)$ comprises the upper blue ray, the magenta circular arc, and the
			lower blue ray and is traversed from top to bottom; 
			$\Omega_t$ is the open set bounded by the magenta and green boundary lines 
			(so these magenta and green lines together form the contour $\Lambda_t'$).
			$\Lambda_1$ comprises the upper blue ray and the upper green line; $\Lambda_2$ denotes
			the union of the lower blue ray and the lower green line, and $\Lambda_3$ is the green circular arc.
		}
	\end{figure}
	Therefore, for each $t \geq 1$, we see that 
	\begin{align}\label{10}
		\mathcal R^\lambda(t) 
		&=  \frac{1}{2\pi i}\int_{\gamma(\frac{\varepsilon}{t},\frac{\pi}{2}+ \delta)} 
						\frac{s^{l(\alpha)- \lambda-1} e^{st}}{Q(s)}ds \nonumber \\ 
		 &=  \frac{1}{2\pi i}\int_{\Lambda_1}\frac{s^{l(\alpha)- \lambda-1} e^{st}}{Q(s)}ds
		 	+ \frac{1}{2\pi i}\int_{\Lambda_2}\frac{s^{l(\alpha)- \lambda-1} e^{st}}{Q(s)}ds
		 	+ \frac{1}{2\pi i}\int_{\Lambda_3}\frac{s^{l(\alpha)- \lambda-1} e^{st}}{Q(s)}ds \nonumber  \\
		&= I_3(t)+ I_4(t)+ I_5(t)
	\end{align}
	with  
	\begin{align*}
		\Lambda_1 & := \left \{ s \in \mathbb C : |s| \geq \frac{\varepsilon}{t}, \arg s = \frac{\pi}{2} + \delta\right \}, \\ 
	 	\Lambda_2 & := \left \{ s \in \mathbb C : |s| \geq \frac{\varepsilon}{t}, \arg s = -(\frac{\pi}{2} + \delta)\right \}, \\ 
	 	\Lambda_3 & := \left \{ s \in \mathbb C : |s| = \frac{\varepsilon}{t}, |\arg s| \leq \frac{\pi}{2} + \delta\right \}. 
	\end{align*}
	Put
	\begin{align}\label{11}
		\eta:=\underset{s\in \gamma (\varepsilon , \frac{\pi}{2} + \delta)\cup B(0,\varepsilon)}{\inf} |Q(s)|. 
	\end{align} 
	For $s \in \Lambda_1,$ $s =re^{i(\frac{\pi}{2} + \delta)} = r(- \sin \delta + i\cos \delta)$ with $r \geq  \frac{\varepsilon}{t}$, and therefore
	\begin{align}
		I_3(t) = \frac{1}{2\pi i}\int_{\varepsilon/t}^{\infty}\frac{r^{l(\alpha)- \lambda - 1}
				e^{i(l(\alpha)- \lambda - 1)(\frac{\pi}{2}+\delta)}
				e^{rt (-\sin \delta + i \cos \delta)}(- \sin \delta + i\cos \delta)}{\widetilde{Q}(r)}dr.  
	\end{align}
	Here, $\widetilde{Q}(r) = Q(re^{i(\frac{\pi}{2}+ \delta)})$. 
	From \eqref{11}, we have the estimate $|\widetilde{Q}(r)| \geq \eta$ for all $r \geq \frac{\varepsilon}{t}$. 
	This implies that
	\begin{align}
		|I_3(t)| \leq  \frac{1}{2\pi \eta}\int_{\varepsilon/t}^{\infty}r^{l(\alpha)- \lambda - 1}e^{-rt\sin \delta }dr . 
	\end{align}
	By the change of variable $r = u/(t \sin \delta)$,
	\begin{align}
		\int_{\varepsilon/t}^{\infty}r^{l(\alpha)- \lambda - 1}e^{-rt\sin \delta }dr 
		& \leq \frac{1}{(t \sin\delta)^{l(\alpha)- \lambda }} \int_{0}^{\infty}u^{l(\alpha)- \lambda - 1}e^{-u}du \nonumber \\ 
		& = \frac{1}{(t \sin\delta)^{l(\alpha)- \lambda }}\Gamma(l(\alpha)- \lambda) \nonumber\\
		& \leq \frac{C_{1,1}}{t^{l(\alpha)- \lambda}}.
	\end{align}
	Hence,
	\begin{align} \label{12}
		 |I_3(t)| \leq \frac{C_{1,1}}{t^{l(\alpha)-\lambda}}.
	\end{align}
	Similarly, there is a $C_{1,2}>0$ such that
	\begin{align} \label{13}
		 |I_4(t)| \leq \frac{C_{1,2}}{t^{l(\alpha)-\lambda }}.  
	\end{align} 
	 For $s \in \Lambda_3,$ $s = (\varepsilon / t) e^{i\varphi}$ with $|\varphi| \leq \frac{\pi}{2} + \delta$, and so
	\begin{align}
		I_5(t) 
		& = - \frac{1}{2 \pi i}\int_{-( \frac{\pi}{2} + \delta)}^{ \frac{\pi}{2} + \delta} 
			\frac{\left (\varepsilon / t  \right)^{l(\alpha)-\lambda -1} e^{i\varphi (l(\alpha)-\lambda -1) } 
				e^{\varepsilon (\cos \varphi+ i \sin \varphi)} i (\varepsilon / t) e^{i\varphi}}{\hat Q(\varphi)} d\varphi 
	\end{align}
	where $\hat Q(\varphi) = Q(\frac{\varepsilon}{t}e^{i\varphi})$. From \eqref{11}, we know that
	$|\hat Q(\varphi)| \geq \eta$ for all $\varphi \in [-(\frac{\pi}{2}+ \delta),\frac{\pi}{2}+ \delta]$. 
	Thus
	\begin{align}\label{14}
		|I_5(t)| 
		& \leq \frac{1}{2 \pi \eta} \left(\frac{\varepsilon}{t}  \right)^{l(\alpha)-\lambda } 
				\int_{-( \frac{\pi}{2} + \delta)}^{ \frac{\pi}{2} + \delta}  e^{\varepsilon \cos \varphi} d\varphi  
		 \leq \frac{C_{1,3}}{t^{l(\alpha) - \lambda}}.
	\end{align}
	From \eqref{10}, \eqref{12}, \eqref{13} and \eqref{14}, we obtain 
	\begin{align}
		 |\mathcal R^\lambda (t)|
		 & \leq \frac{C}{t^{l(\alpha)- \lambda}}
		  \leq \frac{C}{t^\nu}
	\end{align}
	for all $t \ge 1$ and all $\lambda \in \left \{ 0,\alpha_1, \alpha_2 \right \}$, with $C:=C_{1,1}+C_{1,2}+C_{1,3}$.

	\noindent (ii) For the proof of the seocnd statement, we first look at the case $\beta \in \left \{ \alpha_1, \alpha_2 \right \}$.
	Here, we apply the arguments as in the proof of the part (i) above to obtain  
	\begin{align*}
		\mathcal S^\beta(t) 
		&= \frac{1}{2\pi i}\int_{\gamma(\frac{\varepsilon}{t},\frac{\pi}{2}+ \delta)} \frac{s^{l(\alpha)- \beta} e^{st}}{Q(s)}ds \nonumber \\
		&=  \frac{1}{2\pi i}\int_{\Lambda_1}\frac{s^{l(\alpha)- \beta} e^{st}}{Q(s)}ds
			+ \frac{1}{2\pi i}\int_{\Lambda_2}\frac{s^{l(\alpha)- \beta} e^{st}}{Q(s)}ds
			+ \frac{1}{2\pi i}\int_{\Lambda_3}\frac{s^{l(\alpha)- \beta} e^{st}}{Q(s)}ds  
	\end{align*} 
	with each $t \geq 1$.
	In the same way as above, we can find a constant $C_{2,1}$ so that the estimate
	\begin{align}
		|\mathcal S^\beta (t)| \leq \frac{C_{2,1}}{t^{l(\alpha)- \beta + 1}}
	\end{align}
	holds for all $t \ge 1$ and all $ \beta \in \left \{ \alpha_1, \alpha_2 \right \}$.
	Clearly, $t^{l(\alpha)- \beta + 1 } \geq t^{\nu + 1}$ for all $t \geq 1$ and
	all $\beta \in \left \{ \alpha_1, \alpha_2  \right \}$. Thus,
	\begin{align}  
		|\mathcal S^\beta (t)| \leq \frac{C_{2,1}}{t^{\nu + 1}},\; \beta \in \left \{ \alpha_1, \alpha_2 \right \} ,\;  t \geq 1.
	\end{align}
	Next, we consider the remaining case $\beta = l(\alpha)$. For $t \geq 1,$ we see
	\begin{align} \label{15}
		\mathcal S^{l(\alpha)}(t) 
		&= \frac{1}{2\pi i}\int_{\gamma(\frac{\varepsilon}{t},\frac{\pi}{2}+ \delta)} \frac{ e^{st}}{Q(s)}ds \nonumber \\
		&= \frac{1}{2\pi i}\int_{\gamma(\frac{\varepsilon}{t},\frac{\pi}{2}+ \delta)} \frac{1}{\det A}e^{st}ds
			 -  \frac{1}{2\pi i}\int_{\gamma(\frac{\varepsilon}{t},\frac{\pi}{2}+ \delta)}
			 	\frac{(s^{\alpha_1 + \alpha_2} -a_{11}s^{\alpha_2}-a_{22}s^{\alpha_1})e^{st}}{(\det A)Q(s)} ds
			 	\nonumber \\
		&= I_6(t) + I_7(t)
	\end{align}
By using the same estimates as in the proof of the part (i) above, there exist constants $C_{2,2}, C_{2,3}$ and $C_{2,4}$ such that
	\begin{align}\label{bs1}
		|I_7(t)| 
		\leq \frac{C_{2,2}}{t^{l(\alpha) + 1}} +  \frac{C_{2,3}}{t^{\alpha_1 + 1}} +  \frac{C_{2,4}}{t^{\alpha_2 + 1}} 
		\leq  \frac{\sum_{2\leq i\leq 4}C_{2,i}}{t^{\nu + 1}},\; \forall t \geq 1.
	\end{align}
	On the other hand, by the change of variable $s = \frac{u^{1/\mu}}{t}$ with some $\mu \in (0,1)$, we find
	\begin{align}
		I_6(t) 
		&= \frac{1}{2\pi\mu i}\frac{1}{t\det A}\int_{\gamma (\varepsilon^\mu, \mu (\frac{\pi}{2}+ \delta))}
			e^{u^{1/\mu}} u^{(1- \mu)/\mu} du
		=\frac{1}{t\det A}\frac{1}{\Gamma(0)}=0,
	\end{align}
	the last equality being deduced from \cite[eq.~(1.52)]{Podlubny}. 
	Combining \eqref{15} and \eqref{bs1}, for each $t \geq 1$, we conclude
	\begin{align} 
		| \mathcal S^{l(\alpha)}(t)| \leq \frac{C}{t^{\nu +1}}.
	\end{align} 

	\noindent(iii) For each $t \in (0,1)$ and $\beta \in \left \{ \alpha_1,\alpha_2,l(\alpha) \right \}$, we have 
	\begin{align}
		\mathcal S^\beta(t) 
		&=  \frac{1}{2\pi i}\int_{\gamma(R,\frac{\pi}{2}+ \delta)} \frac{s^{l(\alpha)- \beta} e^{st}}{Q(s)}ds \nonumber \\ 
		&=\frac{1}{2\pi i}\int_{\Psi_t }  \frac{s^{l(\alpha)- \beta} e^{st}}{Q(s)}ds 
			+ \frac{1}{2\pi i}\int_{\gamma(\frac{R}{t},\frac{\pi}{2}+ \delta)} \frac{s^{l(\alpha)- \beta} e^{st}}{Q(s)}ds \nonumber \\
		&=I_8(t) + I_9(t) , 
	\end{align}
	where $\Psi _t$ is the boundary of the domain
	$U_t :=\left \{ s \in \mathbb C :  R <|s|  < R / t, |\arg s| < \frac{\pi}{2} + \delta \right \}$. 
	Since $s^{l(\alpha)- \beta} e^{st} / Q(s)$ is analytic on $U_t \cup \Psi_t$ for $t \in (0,1)$, by applying Cauchy's theorem, we obtain
	$I_8(t) = 0$ for all $t  \in (0,1)$.
	Thus,
	\begin{align}\label{20}
		\mathcal S^\beta(t) 
		&=  \frac{1}{2\pi i}\int_{\gamma(\frac{R}{t},\frac{\pi}{2}+ \delta)} \frac{s^{l(\alpha)- \beta} e^{st}}{Q(s)}ds \nonumber \\ 
	 	&= \frac{1}{2\pi i}\int_{\Psi_1}\frac{s^{l(\alpha)- \beta} e^{st}}{Q(s)}ds 
	 		+  \frac{1}{2\pi i}\int_{\Psi_2}\frac{s^{l(\alpha)- \beta} e^{st}}{Q(s)}ds 
	 		+  \frac{1}{2\pi i}\int_{\Psi_3}\frac{s^{l(\alpha)- \beta} e^{st}}{Q(s)}ds \nonumber \\
		& = I_{10}(t) +I_{11}(t) +I_{12}(t)
	\end{align}
	where
	\begin{align*}
		\Psi_1 & := \left \{ s \in \mathbb C : |s| \geq \frac{R}{t}, \arg s = \frac{\pi}{2} + \delta\right \}, \\ 
		\Psi_2 & := \left \{ s \in \mathbb C : |s| \geq \frac{R}{t}, \arg s = -(\frac{\pi}{2} + \delta)\right \}, \\ 
		\Psi_3 & := \left \{ s \in \mathbb C : |s| = \frac{R}{t}, |\arg s| \leq \frac{\pi}{2} + \delta\right \}. 
	\end{align*}
	For $s \in \Psi_1,$  $s =re^{i(\frac{\pi}{2} + \delta)} = r(- \sin \delta + i\cos \delta)$ with $r \geq R/t$, and so
	\begin{align}
		 I_{10}(t) 
		 = \frac{1}{2\pi i}\int_{R/t}^{\infty}
		 	\frac{r^{l(\alpha)- \beta}e^{i(l(\alpha)-\beta)(\frac{\pi}{2} + \delta)}e^{rt(-\sin \delta + i\cos \delta)}}{\widetilde Q (r)} 
		 	(-\sin \delta + i\cos \delta) dr
	\end{align}
	where $\widetilde Q (r) = Q(re^{i(\frac{\pi}{2} + \delta)})$. From \eqref{bs2}, we have the estimate 
	\begin{align}
		|\widetilde Q (r)| 
		= |Q(re^{i(\frac{\pi}{2} + \delta)})| \geq \frac{1}{2}|re^{i(\frac{\pi}{2} + \delta)}|^{\alpha_1 + \alpha_2} 
		= \frac{1}{2}r^{\alpha_1 + \alpha_2},\; \forall r \geq \frac{R}{t}. 
	\end{align}
	This implies 
	\begin{align} \label{17}
		|I_{10}(t)| 
		\leq  \frac{1}{\pi }\int_{R/t}^{\infty}\frac{r^{l(\alpha)- \beta}e^{-rt\sin \delta }}{r^{l(\alpha)}} dr 
		 \leq \frac{1}{\pi t \sin \delta}\int_{R/t}^{\infty}\frac{1}{r^{\beta +1}}dr 
		= \frac{C_{3,1}}{t^{1-\beta}}.
	\end{align}
	The second inequality here is obtained by applying the relation $e^{-x } \leq 1/x$ for $x > 0$.
	Similarly,
	\begin{align} \label{18}
		|I_{11}(t)| 
		&\leq \frac{C_{3,2}}{t^{1-\beta}}. 
	\end{align}
	For $s \in \Psi_3$, $s = (R/t) e^{i\varphi}$ with $|\varphi| \leq \frac{\pi}{2} + \delta$, thus
	\begin{align}
		I_{12}(t) 
		= \frac{1}{2\pi i}\int_{-(\frac{\pi}{2}+ \delta)}^{\frac{\pi}{2}+ \delta} \left( \frac{R}{t} \right)^{l(\alpha) - \beta + 1}
			\frac{e^{i(l(\alpha)- \beta)\varphi}e^{R(\cos \varphi + i \sin \varphi)}}{\hat Q(\varphi)}ie^{i\varphi}d\varphi ,
	\end{align}
	where $\hat Q(\varphi) = Q(\frac{R}{t}e^{i\varphi})$. From \eqref{bs2}, we have 
	\begin{align}
		|\hat Q (\varphi)| 
		= \left| Q \left( \frac{R}{t} e^{i\varphi} \right) \right | 
		\geq \frac{1}{2} \left | \frac{R}{t} e^{i\varphi} \right |^{\alpha_1 + \alpha_2} 
		= \frac{1}{2}\left (\frac{R}{t}  \right )^{\alpha_1 + \alpha_2},\; \forall  \varphi\in [-(\frac{\pi}{2} + \delta),\frac{\pi}{2} + \delta],
	\end{align}
	and thus
	\begin{align}\label{19}
		|I_{12}(t)| 
		&\leq \frac{1}{2\pi} \int_{-(\frac{\pi}{2}+ \delta)}^{\frac{\pi}{2}+ \delta} \left( \frac{R}{t} \right)^{l(\alpha) - \beta + 1}
			\frac{|e^{i(l(\alpha)- \beta)\varphi}| \cdot |e^{R(\cos \varphi + i \sin \varphi)}|}{|\hat Q(\varphi)|}|ie^{i\varphi}|d\varphi 
			\nonumber\\
		&\leq \frac{R^{1-\beta}}{t^{1-\beta}}\frac{1}{\pi}\int_{-(\frac{\pi}{2}+ \delta)}^{\frac{\pi}{2}+ \delta}e^{R\cos \varphi}d\varphi 
		 \leq  \frac{C_{3,3}}{t^{1-\beta}}.
	\end{align}
	From \eqref{20}, \eqref{17}, \eqref{18} and  \eqref{19}, we obtain 
	\begin{align}
		|\mathcal S^{\beta}(t)| \leq \frac{C}{t^{1-\beta}}, \forall t \in (0,1), \beta \in \left \{ \alpha_1,\alpha_2,l(\alpha) \right \}.
	\end{align}

	Finally, \eqref{S3} is an immediate consequence of \eqref{S1} and \eqref{S2}.
\end{proof}

\end{document}